\documentclass[12pt]{amsart}      
\usepackage{txfonts}
\usepackage{amssymb}
\usepackage{eucal}
\usepackage{graphicx}
\usepackage{amsmath}
\usepackage{amscd}
\usepackage[all]{xy}           
\usepackage[active]{srcltx} 
\usepackage{amsfonts,latexsym}
\usepackage{xspace}
\usepackage{epsfig}
\usepackage{float}
\usepackage{color}
\usepackage{fancybox}
\usepackage{colordvi}
\usepackage{multicol}
\usepackage{colordvi}
\usepackage[hypertex]{hyperref} 

\newtheorem{theorem}{Theorem}[section]
\newtheorem{prop}[theorem]{Proposition}
\newtheorem{defn}[theorem]{Definition}
\newtheorem{lemma}[theorem]{Lemma}
\newtheorem{coro}[theorem]{Corollary}
\newtheorem{thm-def}[theorem]{Theorem-Definition}
\newtheorem{def-prop}[theorem]{Definition-Proposition}
\newtheorem{prop-def}[theorem]{Proposition-Definition}
\newtheorem{coro-def}[theorem]{Corollary-Definition}
\newtheorem{claim}{Claim}[section]
\newtheorem{remark}[theorem]{Remark}

\newtheorem{exam}{Example}[section]

\newtheorem{notation}[theorem]{Notation}

\newcommand{\nc}{\newcommand}
\nc{\tred}[1]{\textcolor{red}{#1}}
\nc{\tblue}[1]{\textcolor{blue}{#1}}
\nc{\tgreen}[1]{\textcolor{green}{#1}}
\nc{\tpurple}[1]{\textcolor{purple}{#1}}
\nc{\btred}[1]{\textcolor{red}{\bf #1}}
\nc{\btblue}[1]{\textcolor{blue}{\bf #1}}
\nc{\btgreen}[1]{\textcolor{green}{\bf #1}}
\nc{\btpurple}[1]{\textcolor{purple}{\bf #1}}

\renewcommand{\Bbb}{\mathbb}
\renewcommand{\frak}{\mathfrak}
\newcommand{\efootnote}[1]{}
\renewcommand{\textbf}[1]{}
\newcommand{\delete}[1]{}
\nc{\dfootnote}[1]{{}}          
\nc{\ffootnote}[1]{\dfootnote{#1}}
\nc{\mfootnote}[1]{\footnote{#1}} 
\nc{\ofootnote}[1]{\footnote{\tiny Older version: #1}} 

\delete{
\nc{\mlabel}[1]{\label{#1}  
{\hfill \hspace{1cm}{\bf{{\ }\hfill(#1)}}}}
\nc{\mcite}[1]{\cite{#1}{{\bf{{\ }(#1)}}}}  
\nc{\mref}[1]{\ref{#1}{{\bf{{\ }(#1)}}}}  
\nc{\mbibitem}[1]{\bibitem[\bf #1]{#1}} 
}

\nc{\mlabel}[1]{\label{#1}}  
\nc{\mcite}[1]{\cite{#1}}  
\nc{\mref}[1]{\ref{#1}}  
\nc{\mbibitem}[1]{\bibitem{#1}} 

\nc{\sbar}{\, {\scriptstyle{|\hspace{-.08cm}|\hspace{-.08cm}|
\hspace{-.08cm}|\hspace{-.08cm}|\hspace{-.08cm}|
\hspace{-.08cm}|\hspace{-.08cm}|\hspace{-.08cm}|\hspace{-.08cm}|
\hspace{-.08cm}|\hspace{-.08cm}|\hspace{-.08cm}|
\hspace{-.08cm}|\hspace{-.08cm}|\hspace{-.08cm}|\hspace{-.08cm}|
\hspace{-.08cm}|\hspace{-.08cm}|\hspace{-.08cm}|
\hspace{-.08cm}|\hspace{-.08cm}|\hspace{-.08cm}|\hspace{-.08cm}|
\hspace{-.08cm}|\hspace{-.08cm}|\hspace{-.08cm}|
\hspace{-.08cm}|\hspace{-.08cm}|}\, }}

\nc{\nonu}{{}^{\,,  0}}
\nc{\qnonu}{^{0}}
\nc{\mzvalg}{\mathbf{MZV}} \nc{\edsalg}{\mathbf{EDS}}
\nc{\qsh}{{^{\ast}}}
\nc{\lzero}{_{\hskip -5pt 0}}
\nc{\shf}{{^{\hskip -2pt \ssha}}}
\nc{\shzero}{_{\hskip -7.5pt 0}}
\nc{\lone}{_{\hskip -7.5pt 1}}
\nc{\shqs}{\eta}
\nc{\qssha}{{{\ssha\hspace{-2pt}_\ast}}\,}
\nc{\qsshab}{{{\ssha\hspace{-2pt}_\ast}}\,}
\nc{\un}{u}                 
\nc{\mult}{m}       
\nc{\cprod}{\star}
\nc{\sh}{\mrm{MS}}
\nc{\vr}{{\vec r}} \nc{\vs}{{\vec s}}
\nc{\shprl}{{{\shpr}_\lambda}}

\nc{\wvec}[2]{{\scriptsize{\big [ \!\!
    \begin{array}{c} #1 \\ #2 \end{array} \!\! \big ]}}}
\nc{\lp}{\big ( } \nc{\rp}{\big ) } \nc{\lb}{\!\left \langle }
\nc{\rb}{\right \rangle\! }

\nc{\bin}[2]{ (_{\stackrel{\scs{#1}}{\scs{#2}}})}  
\nc{\binc}[2]{ \big (\!\! \begin{array}{c} \scs{#1}\\
    \scs{#2} \end{array}\!\! \big )}  
\nc{\bincc}[2]{  \left ( {\scs{#1} \atop
    \vspace{-1cm}\scs{#2}} \right )}  
\nc{\lrb}[1]{[#1]}
\nc{\bs}{\bar{S}} \nc{\cosum}{\sqsubset} \nc{\convs}{\frakS_c}
\nc{\conv}{c} \nc{\la}{\longrightarrow} \nc{\rar}{\rightarrow}
\nc{\dar}{\downarrow} \nc{\dap}[1]{\downarrow
\rlap{$\scriptstyle{#1}$}} \nc{\uap}[1]{\uparrow
\rlap{$\scriptstyle{#1}$}} \nc{\defeq}{\stackrel{\rm def}{=}}
\nc{\disp}[1]{\displaystyle{#1}} \nc{\dotcup}{\
\displaystyle{\bigcup^\bullet}\ } \nc{\gzeta}{\bar{\zeta}}
\nc{\hcm}{\ \hat{,}\ } \nc{\hts}{\hat{\otimes}}
\nc{\barot}{{\otimes}} \nc{\free}[1]{\bar{#1}}
\nc{\uni}[1]{\tilde{#1}} \nc{\hcirc}{\hat{\circ}} \nc{\lleft}{[}
\nc{\lright}{]} \nc{\curlyl}{\left \{ \begin{array}{c} {} \\ {}
\end{array}
    \right .  \!\!\!\!\!\!\!}
\nc{\curlyr}{ \!\!\!\!\!\!\!
    \left . \begin{array}{c} {} \\ {} \end{array}
    \right \} }
\nc{\longmid}{\left | \begin{array}{c} {} \\ {} \end{array}
    \right . \!\!\!\!\!\!\!}
\nc{\nvec}[1]{[\vec{#1}]} \nc{\ora}[1]{\stackrel{#1}{\rar}}
\nc{\ola}[1]{\stackrel{#1}{\la}}
\nc{\ot}{\otimes}
\nc{\bigot}{\bigotimes} \nc{\mot}{{{\sbar}}} \nc{\otm}{\bar{\sbar}}
\nc{\scs}[1]{\scriptstyle{#1}} \nc{\subv}{{^{\star}}}
\nc{\cov}{{^{\sharp}}} \nc{\mrm}[1]{{\rm #1}}
\nc{\dirlim}{\displaystyle{\lim_{\longrightarrow}}\,}
\nc{\invlim}{\displaystyle{\lim_{\longleftarrow}}\,}
\nc{\proofbegin}{\noindent{\bf Proof: }}
\nc{\proofend}{$\blacksquare$ \vspace{0.3cm}}
\nc{\sha}{{\mbox{\cyr X}}}  
\newfont{\scyr}{wncyr10 scaled 550}
\nc{\ssha}{{\,\mbox{\bf \scyr X}\,}}
\nc{\shap}{{\mbox{\cyrs X}}} 
\nc{\shpr}{\diamond}    
\nc{\shplus}{\shpr^+}
\nc{\shprc}{\shpr_c}    
\nc{\msh}{\ast} \nc{\vep}{\varepsilon} \nc{\labs}{\mid\!}
\nc{\rabs}{\!\mid}
\nc{\FG}{\mrm{FG}}
\nc{\fp}{\tilde{P}} \nc{\rchar}{\mrm{char}} \nc{\Fil}{\mrm{Fil}}
\nc{\gmzvs}{gMZV\xspace}
\nc{\gmzv}{gMZV\xspace}
\nc{\mzv}{MZV\xspace}
\nc{\mzvs}{MZVs\xspace}
\nc{\MZV}{\mrm{MZV}}
\nc{\Hom}{\mrm{Hom}} \nc{\id}{\mrm{id}} \nc{\im}{\mrm{im}}
\nc{\incl}{\mrm{incl}} \nc{\map}{\mrm{Map}} \nc{\mchar}{\rm char}
\nc{\nz}{\rm NZ} \nc{\supp}{\mathrm Supp}

\nc{\Alg}{\mathbf{Alg}}
\nc{\Bax}{\mathbf{Bax}}
\nc{\bff}{\mathbf f}
\nc{\bfk}{{\bf k}}
\nc{\bfone}{{\bf 1}}
\nc{\bfx}{\mathbf x}
\nc{\bfy}{\mathbf y}
\nc{\base}[1]{\bfone^{\otimes ({#1}+1)}} 
\nc{\Cat}{\mathbf{Cat}}

\nc{\detail}{\marginpar{\bf More detail}
    \noindent{\bf Need more detail!}
    \smallskip}
\nc{\Int}{\mathbf{Int}}
\nc{\Mon}{\mathbf{Mon}}
\nc{\remarks}{\noindent{\bf Remarks: }}
\nc{\Rings}{\mathbf{Rings}}
\nc{\Sets}{\mathbf{Sets}}

\nc{\BA}{{\Bbb A}} \nc{\CC}{{\Bbb C}} \nc{\DD}{{\Bbb D}}
\nc{\EE}{{\Bbb E}} \nc{\FF}{{\Bbb F}} \nc{\GG}{{\Bbb G}}
\nc{\HH}{{\Bbb H}} \nc{\LL}{{\Bbb L}} \nc{\NN}{{\Bbb N}}
\nc{\KK}{{\Bbb K}} \nc{\QQ}{{\Bbb Q}} \nc{\RR}{{\Bbb R}}
\nc{\TT}{{\Bbb T}} \nc{\VV}{{\Bbb V}} \nc{\ZZ}{{\Bbb Z}}

\nc{\cala}{{\mathcal A}} \nc{\calc}{{\mathcal C}}
\nc{\cald}{{\mathcal D}} \nc{\cale}{{\mathcal E}}
\nc{\calf}{{\mathcal F}} \nc{\calg}{{\mathcal G}}
\nc{\calh}{{\mathcal H}} \nc{\cali}{{\mathcal I}}
\nc{\call}{{\mathcal L}} \nc{\calm}{{\mathcal M}}
\nc{\caln}{{\mathcal N}} \nc{\calo}{{\mathcal O}}
\nc{\calp}{{\mathcal P}} \nc{\calr}{{\mathcal R}}
\nc{\cals}{{\mathcal S}}
\nc{\calt}{{\mathcal T}} \nc{\calw}{{\mathcal W}}
\nc{\calk}{{\mathcal K}} \nc{\calx}{{\mathcal X}}
\nc{\CA}{\mathcal{A}}

\nc{\fraka}{{\frak a}}
\nc{\frakA}{{\frak A}}
\nc{\frakb}{{\frak b}}
\nc{\frakB}{{\frak B}}
\nc{\frakH}{{\frak H}}
\nc{\frakL}{{\mathfrak L}}
\nc{\frakM}{{\frak M}}
\nc{\bfrakM}{\overline{\frakM}}
\nc{\frakm}{{\frak m}}
\nc{\frakP}{{\frak P}}
\nc{\frakN}{{\mathfrak N}}
\nc{\frakp}{{\frak p}}
\nc{\frakR}{{\frak R}}
\nc{\frakS}{{\frak S}}
\nc{\frakx}{{\mathfrak x}}

\font\cyr=wncyr10
\font\cyrs=wncyr7
\nc {\p}{\partial }
\nc {\lne}{{\ln (-\varepsilon)}}
\nc{\dl}{{\delta}}
\nc{\redt}[1]{\textcolor{red}{#1}}
\nc{\li}[1]{\textcolor{red}{Li:#1}}
\nc{\zb}[1]{\textcolor{blue}{Bin: #1}}
\nc{\zhb}[1]{\textcolor{red}{Bin: #1}}
\begin{document}

\title[Polylogarithms, MZVs and Rota-Baxter algebras]{Polylogarithms and multiple zeta values from free Rota-Baxter algebras}
%
\author{Li Guo}
\address{Department of Mathematics and Computer Science,
         Rutgers University,
         Newark, NJ 07102}
\email{liguo@rutgers.edu}
\author{Bin Zhang}
\address{Yangtze Center of Mathematics,
Sichuan University, Chengdu, 610064, P. R. China }
\email{zhangbin@scu.edu.cn}

\date{May 7, 2010}

\begin{abstract}
We show that the shuffle algebras for polylogarithms and regularized MZVs in the sense of Ihara, Kaneko and Zagier are both free commutative nonunitary Rota-Baxter algebras with one generator. We apply these results to show that the full sets of shuffle relations of polylogarithms and regularized MZVs are derived by a single series. We also take this approach to study the extended double shuffle relations of MZVs by comparing these shuffle relations with the quasi-shuffle relations of the regularized MZVs in our previous approach of MZVs by renormalization.
\end{abstract}

\maketitle



\tableofcontents

\setcounter{section}{0}


\section {Introduction}
In this paper we show that there is a free commutative nonunitary
Rota-Baxter algebra with one generator behind the shuffle relations
of polylogarithms and regularized multiple zeta values. Thus by the
universal property of a free Rota-Baxter algebra, with a suitable
choice of the image for this generator, one can recover all the
shuffle relations of polylogarithms and regularized multiple zeta
values. We also apply this approach to derive the extended double
shuffle relation in the work of Ihara, Kaneko and Zagier.

{\bf Multiple zeta values (MZVs)} are defined to be the evaluation of the multiple complex variable function
\begin{equation}
\zeta(\vec{s})=\zeta(s_1,\cdots, s_k)=\sum_{n_1>\cdots>n_k\geq 1}
    \frac{1}{n_1^{s_1}\cdots n_k^{s_k}}
\mlabel{eq:mzv}
\end{equation}
at positive integers $s_1,\cdots,s_k$ with $s_1>1$.
They were introduced in the early 1990s with motivation from number theory, combinatorics and quantum field theory~\mcite{Ho0,BK,Za}.
Since then
the subject has been studied intensively with interactions with a broad range of areas in mathematics and physics, including arithmetic geometry, combinatorics, number theory, knot theory, Hopf algebra, quantum field theory and mirror symmetry~\mcite{An,3BL,Ca1,GM,IKZ,Ra,Te}.
Many mathematicians and physicists have contributed to this area, including Broadhurst, Cartier, Deligne, Goncharov, Hoffman, Kontsevich, Kreimer, Manin and Zagier.

MZVs can be regarded as basic blocks of important invariants in mathematics and physics.
In mathematics, periods of all mixed Tate motives are conjecturally rational linear combinations of MZVs. In physics, most computed values of Feynman integrals in quantum field theory are also linear combinations of MZVs.
A basic conjecture on MZVs is the Dimension Conjecture of Zagier which implies that $\zeta(2), \zeta(n)$, for $n\geq 2$ odd, are algebraically independent over $\mathbb{Q}$. Further number theoretic significance of MZVs comes from their connection with mixed Tate motives, motivic Galois groups, modular forms and Drinfeld associators.

A special role played by special values of various one variable zeta
functions and $L$-functions in number theory is to make the
connection between the analytic theory and the algebraic theory. In
a similar spirit, MZVs also have the remarkable property that these
purely analytic defined values have a purely algebraically defined
structure determined by the double shuffle product. More precisely,
there are many linear relations and in general algebraic relations. For example,
$$ \zeta(3)=\zeta(2,1), \quad \zeta(4)=4\zeta(3,1)$$
which are already known to Euler. Most of the found relations among
MZVs come from the extended double shuffle relation as a combination
of the shuffle and stuffle (quasi-shuffle) relations of MZVs. In fact, it is conjectured that all algebraic relations of MZVs can be obtained
this way~\mcite{Ra,IKZ}. Therefore the extended double shuffle relations give a well-formulated algebraic
framework to study the analytically defined MZVs.

Thus it is important to understand the extended double shuffle relations. The purpose of this paper is to give more structures on these relations from the point of view of free Rota-Baxter algebras~\mcite{Ba,GK1,Ro}. We also consider the shuffle relation of multiple polylogarithms~\mcite{3BL,Go}.

Recently, there have been a lot of interest to study MZVs
$\zeta(s_1,\cdots,s_k)$ beyond the region $s_1>1, s_i\geq 1, 1\leq
i\leq k$ where they converge. Since analytic continuation fails to
define most of these values, other approaches have been adopted,
such as taking directional limits and renormalization. The
approach of renormalization was introduced from the renormalization
of quantum field theory in the algebraic framework of Connes and
Kreimer~\mcite{CK,CK1}. It works quite well to define MZVs at
negative arguments that extend the quasi-shuffle relations of the convergent MZVs~\mcite{GZ,GZ2,MP2,Zh2}. This paper
comes from our effort in extending the shuffle relation.

After reviewing the background on free Rota-Baxter algebras and their construction by mixable shuffle products, we establish in Section~\mref{sec:sh1} the freeness of the nonunitary shuffle algebra $\calh\shf\lone\nonu$ for MZVs. This property is extended to a larger free Rota-Baxter algebra $\calh_{\geq 0}\qnonu$ in Section~\mref{sec:sh0}. The applications are given in Section~\mref{sec:app}. We first apply the Rota-Baxter algebraic freeness of $\calh_{\geq 0}\qnonu$ to recover the shuffle relation of multiple polylogarithms~\mcite{3BL,Go}. We then apply the Rota-Baxter algebraic freeness of $\calh\shf\lone\nonu$ to recover the extended shuffle relation of regularized MZVs in the sense of Ihara, Kaneko and Zagier~\mcite{IKZ}. We finally show that this freeness property of $\calh\shf\lone\nonu$, together with the regularized MZVs in our renormalization approach~\mcite{GZ,GZ2}, provides another way to obtain the map $\rho$ of~\mcite{IKZ} that serves as the key link for the extended double shuffle relation.

\smallskip

\noindent {\bf Acknowledgements: } L. Guo acknowledges the support
from NSF grant DMS-0505643 and thanks  the Center of Mathematics at
Zhejiang University for its hospitality. B. Zhang acknowledges the
support from NSFC grant 10631050 and 10911120391/A0109.

\section{Free Rota-Baxter algebra structure on the shuffle algebra $\calh\shf\lone\nonu$}
\mlabel{sec:sh1}
We start with a summary of the related background on Rota-Baxter algebra. We then prove the one generator freeness of the shuffle algebra $\calh\shf\lone\nonu$ that has arisen from the study of MZVs.

\subsection{Rota-Baxter algebras and shuffle products}
To provide the necessary
motivation and background for our study, we briefly review Rota-Baxter algebras and their free objects in the commutative case. For further
details, see the survey papers~\mcite{Gugn,Gwi,GPXZ} and the references therein.

All rings and algebras in this paper are assumed to be unitary unless otherwise specified. Let $\bfk$ be a commutative ring whose identity is denoted by $1$.
\subsubsection{Rota-Baxter algebras}
Rota-Baxter algebra is an abstraction of the algebra of
continuous functions acted by the integral operator. It
originated from the probability study of Glenn Baxter~\mcite{Ba} in 1960
and was developed further by Cartier and the school of Rota
in the 1960s and 1970s~\mcite{Ca,Ro}. Independently, this structure
appeared in the Lie algebra context as the operator form of
the classical Yang-Baxter equation in the 1980s~\mcite{STS,Bai}.
Since the late 1990s, Rota-Baxter algebra has found important
theoretical developments and applications in mathematical physics, operads, number theory and combinatorics~\mcite{Ag,BGN,CK,EG1,EGK,EGM,GK1,GZ}.

\begin{defn}
{\rm
Let $\lambda\in \bfk$ be fixed. A unitary (resp. nonunitary) {\bf Rota--Baxter
$\bfk$-algebra {(RBA)} of weight $\lambda$} is a pair $(R,P)$
consisting of unitary (resp. nonunitary) $\bfk$-algebra $R$ and a $\bfk$-linear map
$P: R \to R$ such that
\begin{equation}
 P(x)P(y) = P(xP(y))+P(P(x)y)+ \lambda P(xy),\ \forall x,\ y\in R.
\mlabel{eq:Ba}
\end{equation}
Then $P$ is called a {\bf Rota-Baxter operator}.
}
\mlabel{de:rba}
\end{defn}

A Rota-Baxter algebra homomorphism $f:(R,P)\to (R',P')$ between Rota-Baxter $\bfk$-algebras $(R,P)$ and $(R',P')$ is a $\bfk$-algebra homomorphism $f:R\to R'$ such that $f\circ P = P'\circ f$.

The following Rota-Baxter operators have played important roles in the study of MZVs.
\begin{exam}{\rm
{\bf (The integration operator)}
Let $R$ be the $\RR$-algebra $C[0,\infty)$ of continuous functions $f(x)$ on $[0,\infty)$.
Then the integration operator
\begin{equation}
 P:R\to R, \quad P(f)(x)=\int_0^x f(t)dt
 \mlabel{eq:int}
 \end{equation}
is a Rota-Baxter operator of weight $0$ by the integration by parts formula~\mcite {Ba}. A variation of this operator is the operator $J$ in Eq.~(\mref{eq:logint}).
} \mlabel{ex:int}
\end{exam}
\begin{exam} {\rm
{\bf (The summation operator)}
Consider the summation operator~\mcite{Zud}
$$P(f)(x):= \sum_{n\geq 1} f(x+n)$$
on functions with suitable convergency conditions, such as $f(x) = O(x^{-2})$. It is a Rota-Baxter operator of weight 1.
}
\mlabel{ex:sum}
\mlabel{ex:ps}
\end{exam}

\begin{exam} {\rm
{\bf (The pole part projector)}
Let $A=\bfk[\vep^{-1},\vep]]$ be the algebra of Laurent series. Define $\Pi:A\to A$ by
\begin{equation}
\Pi\big(\sum_{n} a_n \vep^n \big)=\sum_{n<0} a_n \vep^n.
\mlabel{eq:lau}
\end{equation}
Then $\Pi$ is a Rota-Baxter operator of weight $-1$. This operator arises in the renormalization of quantum field theory and multiple zeta values~\mcite{CK,GZ,MP2}.
}
\mlabel{ex:lau}
\end{exam}

\subsubsection{Shuffle products and free Rota-Baxter algebras} \mlabel{ss:msh}
We briefly recall the construction of shuffle and quasi-shuffle products in the framework of mixable shuffle algebras~\mcite{GK1,GK2}.

Let $\bfk$ be a commutative ring.
Let $A$ be a commutative $\bfk$-algebra {\em that is not necessarily unitary}. For a given $\lambda\in \bfk$, the {\bf mixable shuffle algebra of weight $\lambda$ generated by $A$} (with coefficients in $\bfk$) is $\sh(A)=\sh_{\bfk,\lambda}(A)$ whose underlying $\bfk$-module is that of the tensor algebra
\begin{equation}
T(A)= \bigoplus_{k\ge 0}
    A^{\otimes k}
= \bfk \oplus A\oplus A^{\otimes 2}\oplus \cdots
\mlabel{eq:mshde}
\end{equation}
equipped with the {\bf
mixable shuffle product $\shprl$ of weight $\lambda$} defined as
follows.

For pure tensors
$\fraka=a_1\ot \ldots \ot a_k\in A^{\ot k}$ and $\frakb=b_1\ot
\ldots \ot b_\ell\in A^{\ot \ell}$, a {\bf shuffle} of $\fraka$ and
$\frakb$ is a tensor list of $a_i$ and $b_j$ without change the
natural orders of the $a_i$s and the $b_j$s. The {\bf shuffle product} $\fraka \ssha \frakb$ is the sum of all shuffles of $\fraka$ and $\frakb$. The product can also be defined recursively by
$$
\fraka \ssha \frakb:= a_1\ot \big((a_2\ot \cdots\ot a_k)\ssha \frakb\big) + b_1\ot \big(\fraka \ssha (b_2\ot \cdots \ot b_\ell)\big)
$$
with the convention that if $k=1$ (resp. $\ell=1$) then $a_2\ot \cdots\ot a_k$ (resp. $b_2\ot \cdots \ot b_\ell$) is the identity.

More generally, for a
fixed $\lambda\in \bfk$, a {\bf mixable shuffle} (of weight
$\lambda$) of $\fraka$ and $\frakb$ is a shuffle of $\fraka$ and
$\frakb$ in which some (or {\it none}) of the pairs $a_i\ot b_j$ are
merged into $\lambda\, a_i b_j$. Then the {\bf mixable shuffle product (of weight $\lambda$)} $\fraka \shprl \frakb$ is defined to be the sum of mixable shuffles of $\fraka$ and $\frakb$. When $\lambda=0$, we simply have the shuffle product which is also defined when $A$ is only a $\bfk$-module.

The product $\shprl$ can also be defined by the following
recursion~\mcite{EGsh,GZ2,Ho2,MP2}.
\begin{equation}
  \fraka \shprl \frakb =
  a_1\ot  \big ((a_2\ot \cdots\ot a_k)\shprl
\frakb \big ) + b_1\ot  \big (\fraka \shprl (b_2 \ot \cdots \ot  b_\ell)\big)  + \lambda(a_1 b_1)  \big ( (a_2\ot \cdots\ot a_k) \shprl
     (b_2\ot \cdots\ot  b_\ell)\big )
\mlabel{eq:quasi}
\end{equation}
with the convention that if $k=1$ (resp. $\ell=1$) then $a_2\ot \cdots\ot a_k$ (resp. $b_2 \ot \cdots \ot  b_\ell$) is the identity. Further, if $k=\ell=1$ then take $\lambda(a_1 b_1)  \big ( (a_2\ot \cdots\ot a_k) \shprl
     (b_2\ot \cdots\ot  b_\ell)\big )=\lambda (a_1 b_1)$.

We have the following relationship between mixable shuffle product and free commutative Rota-Baxter algebras.

\begin{theorem} {\rm (}\mcite{GK1}{\rm )} The tensor product algebra $\sha(A):=\sha_{\bfk,\lambda}(A)= A\ot \sh_{\bfk,\lambda}(A)$, with the linear operator $P_A:\sha(A)\to \sha(A)$ sending $\fraka$ to $1\ot \fraka$, is the free commutative Rota-Baxter algebra of weight $\lambda$ generated by A.
\mlabel{thm:freea}
\end{theorem}

Now let $A$ be a commutative nonunitary $\bfk$-algebra and let $\uni{A}=\bfk\oplus A$ be the unitarization of $A$. Define
\begin{equation}
 \sha_\bfk(A)^0=\bigoplus_{n\geq 0} (\uni{A}^{\ot n}\ot A)
\mlabel{eq:freenrb}
\end{equation}
with the convention that $\uni{A}^{\ot 0}=\bfk$ and thus $\uni{A}^{\ot 0}\ot A=A$.
Then $\sha_\bfk(A)^0$ is the $\bfk$-submodule of $\sha_\bfk(\uni{A})$,
additively spanned by tensors of the form
\[ a_0 \otimes\ldots \otimes a_n,\quad
a_i\in \uni{A}, 0\leq i\leq n-1,\ a_n\in A.\]
Then $\sha_\bfk(A)^0$, with the restriction of $P_{\uni{A}}$,
denoted by $P_A$, is a subobject of $\sha_\bfk(\uni{A})$ in the category of commutative non-unitary Rota-Baxter algebras.
By Proposition 2.6 of~\mcite{GK2}, $(\sha_\bfk(A)^0,P_A)$ is the free commutative non-unitary Rota-Baxter algebra generated by $A$. In the rest of the paper, we will be most interested in the special case when $A=x\bfk[x]$ and thus $\uni{A}=\bfk[x]$ and when the weight $\lambda$ is 0.
We make the statement precise for the convenience of later references.

\begin{theorem} $($\cite[Proposition 2.6]{GK2}$)$ Denote
$$\sha(x\bfk[x])^0=\bigoplus_{u_i\geq 0, 1\leq i\leq k, u_k\ge1, k\geq 1} \bfk x^{u_1}\ot \cdots \ot x^{u_k}.$$
Then with the restriction of the product and the Rota-Baxter operator $P_x$ in $\sha(\bfk[x])=\sha_{\bfk,0}(\bfk[x])$, $\sha(x\bfk[x])^0$ is the free commutative nonunitary Rota-Baxter algebra of weight $0$ generated by $x$. More precisely, for any commutative nonunitary Rota-Baxter algebra of weight 0 $(R,P)$ and a given element $r\in R$, there is a unique homomorphism of non-unitary Rota-Baxter algebras of weight 0 $f:\sha(x\bfk[x])^0 \to R$ such that $f(x)=r$.
\mlabel{thm:free}
\end{theorem}
\begin{remark}
{\rm
For the rest of the paper, we will only consider free commutative nonunitary Rota-Baxter algebras of weight 0. So the term weight 0 will sometimes be suppressed for notational simplicity.
}
\mlabel{rk:weight}
\end{remark}

Let $G$ be a semigroup and let $\bfk\,G=\sum_{g\in G} \bfk\,g$ be the semigroup nonunitary $\bfk$-algebra. A canonical $\bfk$-basis of $(\bfk\,G)^{\ot k}, k\geq 0$, is the set $G^{\ot k}:=\{g_1\ot \cdots \ot g_k\ |\  g_i\in G, 1\leq i\leq k\}$.
Let $G$ be a graded semigroup $G=\coprod_{i\geq 0} G_i$, $G_iG_j\subseteq G_{i+j}$ such that $|G_i|<\infty$, $i\geq 0$.
Then the mixable shuffle product $\shpr_1$ of weight $1$ is identified with the {\bf quasi-shuffle product} $\ast$ defined by Hoffman~\mcite{Ho2,EGsh,GZ2}.
\begin{notation}
{\rm
\begin{enumerate}
\item
To simplify the notation and to be consistent with the conventions in the literature of MZVs, we will identify $g_1\ot \cdots \ot g_k$ with the concatenation $g_1\cdots g_k$ unless there is a danger of confusion.
We also denote the weight $1$ mixable shuffle product $\shpr_1$ by $\ast$ and denote the corresponding mixable algebra $\sh_{\bfk,1}(A)$ by $\calh_A^\ast$. Similarly, when $A$ is taken to be a $\bfk$-module, we denote the weight zero mixable shuffle algebra $\sh_{\bfk,0}(A)$ by $\calh_A^\shf$.
\item
Further, if our multiplicatively defined semigroup $(G,\cdot)$ comes from an additive semigroup $S$ in the sense that $G=G_S:=\{[s]\ |\ s\in S\}$ such that $[s]\cdot [s']=[s+s']$. We then let $[s_1,\cdots,s_k]$ denote $[s_1]\cdots[s_k]$ (which is abbreviated from $[s_1]\ot \cdots \ot [s_k]$ by the previous notation).
This applies in particular to the case when $G$ is taken to be
\begin{equation}
G_{\geq n}:=G_{\ZZ_{\geq n}} \text{ where }
\ZZ_{\geq n}=\{s\in \ZZ\ |\ s\geq n\}, \quad n=0,1.
\mlabel{eq:g01}
\end{equation}
We will use the notation
\begin{equation}
\calh_{\geq n}=\calh_{\ZZ G_{\geq n}}=\bigoplus_{s_i\geq n, 1\leq i\leq k, k\geq 0} \bfk [s_1,\cdots,s_k], \quad
\calh_{\geq n}\qnonu=\bigoplus_{s_i\geq n, 1\leq i\leq k, k\geq 1} \bfk [s_1,\cdots,s_k], \quad n=0,1.
\mlabel{eq:qsh01}
\end{equation}
\end{enumerate}
}
\mlabel{no:mscase}
\end{notation}

\subsection{Rota-Baxter algebra freeness of shuffle algebras}
\mlabel{ss:free1}
In this section, we study the freeness of the shuffle algebra
$\calh\shf\lone$ for MZVs in the category of Rota-Baxter algebras.

Consider the set $X=\{x_0,x_1\}$. With the convention in
Notation~\mref{no:mscase}, we denote the shuffle algebra
$\calh\shf:=\calh\shf_{\hspace{-.2cm}\QQ\,X}$ whose underlying module is $\QQ\langle x_0,x_1\rangle$ (the noncommutative polynomial algebra) and which contains
the following nonunitary subalgebra
\begin{equation}
 \calh\shf\lone\nonu:= \calh\shf x_1
=\QQ\langle x_0, x_1\rangle x_1
=\bigoplus_{u_i\geq 0, 1\leq i\leq k, k\geq 0} \QQ x_0^{u_1}x_1x_0^{u_2}x_1\cdots x_0^{u_k}x_1.
\mlabel{eq:h1}
\end{equation}
Its unitarization is $\calh\shf\lone:=\QQ \oplus \calh\shf x_1.$

We now prove our first theorem on free Rota-Baxter algebras.
\begin{theorem}
The nonunitary shuffle algebra $\calh\shf\lone\nonu$ in Eq.~(\mref{eq:h1}), together with the left multiplication operator $I_0(w)=x_0w$, is the free commutative nonunitary Rota-Baxter algebra of weight 0 generated by $x_1$.
\mlabel{thm:free1}
\end{theorem}

\begin{proof}
We first prove that $\calh\shf\lone\nonu$ is generated by $x_1$ as a nonunitary Rota-Baxter algebra. Let $R'$ be the nonunitary Rota-Baxter subalgebra of $\calh\shf\lone\nonu$ generated by $x_1$. By Eq.~(\mref{eq:h1}) we only need to show that
$x_0^{u_1}x_1x_0^{u_2}x_1\cdots x_0^{u_k}x_1$ is in $R'$ for all $u_i\geq 0, 1\leq i\leq k, k\geq 1$. Since
$$ x_0^{u_1}x_1x_0^{u_2}x_1\cdots x_0^{u_k}x_1= I_0^{u_1}(x_1x_0^{u_2}x_1\cdots x_0^{u_k}x_1),$$
we only need to show
\begin{claim}
$w:=x_1x_0^{u_2}x_1\cdots x_0^{u_k}x_1$ is in $R'$ for all $u_i\geq 0, 2\leq i\leq k, k\geq 1$.
\mlabel{cl:ind1}
\end{claim}
For this purpose, we apply induction on $k$. When $k=1$, we have $w=x_1$ which is in $R'$ by assumption. Suppose that Claim~\mref{cl:ind1} has been proved for $k=a\geq 1$. It remains to prove
\begin{claim}
$w=x_1x_0^{u_2}x_1\cdots x_0^{u_{a+1}}x_1$ is in $R'$.
\mlabel{cl:ind2}
\end{claim}
We prove Claim~\mref{cl:ind2} by a second induction on $m:=u_2+\cdots+u_{a+1}\geq 0$. When $m=0$, we have $u_2=\cdots =u_{a+1}=0$ and so $w=x_1^{a+1}$ which equals $\frac{1}{(a+1)!}x_1^{\ssha (a+1)}$ which is in $R'$. Assume that Claim~\mref{cl:ind2} has been proved for $m=b\geq 0$. It remains to prove
\begin{claim}
Any $w=x_1x_0^{u_2}x_1\cdots x_0^{u_{a+1}}x_1$ with $m=b+1$ is in $R'$.
\mlabel{cl:ind3}
\end{claim}
We prove Claim~\mref{cl:ind3} by a third induction on $n\geq 2$ such that $u_2 =\cdots =u_{n-1}=0$ and $u_n>1$. When $n=2$, we have $u_2>1$. Then by the definition of the shuffle product, we have
\begin{eqnarray}
 x_1 \ssha x_0^{u_2}x_1\cdots x_0^{u_{a+1}}x_1 &=&
x_1 (1 \ssha x_0^{u_2}x_1\cdots x_0^{u_{a+1}}x_1) + x_0(x_1 \ssha (x_0^{u_2-1}x_1\cdots x_0^{u_{a+1}}x_1) \notag \\
 &=& x_1x_0^{u_2}x_1\cdots x_0^{u_{a+1}}x_1 + x_0 w_1+ \cdots +x_0 w_r, \mlabel{eq:ind3}
\end{eqnarray}
where $w_1,\cdots,w_r$ are of the form $x_0^{v_1}x_1x_0^{v_2}x_1\cdots x_0^{v_{a+1}}x_1$ with $v_1,\cdots,v_{a+1}\geq 0$ and $v_2+\cdots+v_{a+1}\leq b$. Hence by the induction hypothesis for the third induction, we have $x_1x_0^{v_2}x_1\cdots x_0^{v_{a+1}}x_1\in R'$ and hence
$x_0^{v_1}x_1x_0^{v_2}x_1\cdots x_0^{v_{a+1}}x_1=
I_0^{v_1}(x_1x_0^{v_2}x_1\cdots x_0^{v_{a+1}}x_1)\in R'$. Thus $v_i$ and hence $x_0v_i=I_0(v_i)$ are in $R'$ for $1\leq i\leq r$.  Since $x_1$ is in $R'$ by the definition of $R'$ and $x_0^{u_2}x_1\cdots x_0^{u_{a+1}}x_1$ is in $R'$ by the induction hypothesis of the first induction, from Eq.~(\mref{eq:ind3}) we conclude that $x_1x_0^{u_2}x_1\cdots x_0^{u_{a+1}}x_1$ is in $R'$. This completes the third induction and proves Claim~\mref{cl:ind3}, which in turns completes the second induction and proves Claim~\mref{cl:ind2}, which in turn completes the first induction and proves Claim~\mref{cl:ind1}. Thus $\calh\shf\lone\nonu$ is a nonunitary Rota-Baxter algebra generated by $x$.

\smallskip

By the universal property of the free commutative nonunitary Rota-Baxter algebra $\sha(x\bfk[x])^0$ in Theorem~\mref{thm:free}, we have a homomorphism
$$ f: \sha(x\bfk[x])^0 \to \calh\shf\lone\nonu$$
of nonunitary Rota-Baxter algebras such that $f(x)=x_1$. Since we have shown that $\calh\shf\lone\nonu$ is a nonunitary Rota-Baxter algebra generated by $x$, $f$ is surjective.
Thus to prove the theorem, it remains to show that $f$ is injective.

First note that, for $\frakx:=x^{n_1}\ot \cdots \ot x^{n_k}\in \sha(x\bfk[x])^0$ with $n_k\geq 1, n_i\geq 0, 1\leq i\leq k$, we have
\begin{equation}
\frakx=x^{n_1} \shpr P_x(x^{n_2}\shpr P_x(\cdots P_x(x^{n_k})\cdots )).
\mlabel{eq:map1}
\end{equation}
Thus
\begin{equation}
f(\frakx)= x_1^{\ssha n_1} \ssha (x_0f(x^{n_2}\ot \cdots \ot x^{n_k}))=
x_1^{\ssha n_1} \ssha (x_0(x^{\ssha n_2} \ssha (x_0 (\cdots \ssha (x_0(x_1^{\ssha n_k})))))).
\mlabel{eq:map2}
\end{equation}

We next define gradings on $\sha(x\bfk[x])^0$ and on $\calh\shf\lone\nonu$ that make them graded algebras. For $\frakx=x^{n_1}\ot \cdots \ot x^{n_k}$ with $n_1\geq 1, n_i\geq 0, 1\leq i\leq k$, define $$ \deg(\frakx)=n_1+\cdots+n_k+k-1.$$
This defines a grading on $\sha(x\bfk[x])^0$. Let $\sha(x\bfk[x])^0_m$ be the $m$-th homogeneous subspace of $\sha(x\bfk[x])^0$. A basis of $\sha(x\bfk[x])^0_m$ consists of the elements $\frakx:=x^{n_1}\ot \cdots \ot x^{n_k}$ of $\sha(x\bfk[x])^0_m$ with $n_1+\cdots+n_k+k-1=m$.
Such an element can be uniquely determined from a string of $m-1$ $x$'s by replacing $0\leq i\leq m-1$ of the $x$'s by the tensor symbol $\ot$ and then amending an $x$ factor to the end. Thus there are
$$ \binc{m-1}{0}+\cdots +\binc{m-1}{m-1} = 2^{m-1}$$
such elements and $\dim(\sha(x\bfk[x])^0_m)=2^{m-1}$.

Similarly, for $x_0^{u_1}x_1\cdots x_0^{u_k}x_1$ with $u_i\geq 0, 1\leq i\leq k, k\geq 1$, define
$$\deg(x_0^{u_1}x_1\cdots x_0^{u_k}x_1)=u_1+\cdots+u_k+k.$$
This defines a grading on $\calh\shf\lone\nonu$. Let $\calh_m$ be the $m$-th homogenous subspace of $\calh\shf\lone\nonu$. A basis of $\calh_m$ consists of elements of the form $x_0^{u_1}x_1\cdots x_0^{u_k}x_1$ with $u_1+\cdots+u_k + k=m$. Such an element is uniquely determined from a string of $m-1$ $x_0$'s by replacing $0\leq i\leq m-1$ of the $x_0$'s by $x_1$'s and then amending an $x_1$ to the end. Thus there are also $2^{m-1}$ such basis elements and $\dim (\calh_m)=2^{m-1}$.

We note that, for $\frakx$ in Eq.~(\mref{eq:map1}), $\deg(\frakx)$ is the total number of $x$ and $P_x$ on the right hand side of the equation. By Eq.~(\mref{eq:map2}), the map $f$ converts each $x$ to an $x_1$ and each $P_x$ to an $x_0$. Thus $f:\sha(x\bfk[x])^0\to \calh\shf\lone\nonu$ is a graded algebra homomorphism. Hence $f$ restricts to $f_m:\sha(x\bfk[x])^0_m \to \calh_m$, $m\geq 1$. Since $f$ and hence $f_m$ is surjective and the dimensions of $\sha(x\bfk[x])^0_m$ and $\calh_m$ are the same, the linear map $f_m$ must be bijective. Thus $f$ is bijective and the proof of the theorem is completed.
\end{proof}

\section{Free Rota-Baxter algebra structure on the shuffle algebra $\calh_{\geq 0}\qnonu$}
\mlabel{sec:sh0}
We next show that the free Rota-Baxter algebra structure on the shuffle algebra $(\calh\shf\lone\nonu,\ssha)$ in fact comes from (i.e., is the restriction of) a larger shuffle algebra $\calh_{\geq 0}\qnonu$ which is also a free Rota-Baxter algebra with one generator.
We first rephrase in Section~\mref{ss:shqsh} the free Rota-Baxter algebra structure on $\calh\shf\lone\nonu$ in terms of $\calh_{\geq 1}\qnonu$, naturally a subset of $\calh_{\geq 0}\qnonu$. Then in Section~\mref{ss:sh0}, we extend this free Rota-Baxter algebra structure on $\calh_{\geq 1}\qnonu$ to a Rota-Baxter algebra structure on $\calh_{\geq 0}\qnonu$, and show in Section~\mref{ss:free0} that this Rota-Baxter algebra on $\calh_{\geq 0}\qnonu$ is free.

\subsection{Free Rota-Baxter algebra on $\calh_{\geq 1}\qnonu$}
\mlabel{ss:shqsh}
Recall the notations from Notation~\mref{no:mscase}:
$$
\calh_{\geq n}\qnonu=\bigoplus_{s_i\geq n, 1\leq i\leq k, k\geq 1} \bfk [s_1,\cdots,s_k], \quad n=0,1.
$$
The map
\begin{equation}
\eta: \calh\shf\lone\nonu \to \calh_{\geq 1}\qnonu, \quad x_0^{s_1-1}x_1\cdots x_0^{s_k-1}x_1 \mapsto [s_1,\cdots,s_k].
\mlabel{eq:shqsh}
\end{equation}
defines a bijection. By transporting of structures, from the Rota-Baxter algebra $(\calh\shf\lone\nonu,\ssha,I_0)$ in Theorem~\mref{thm:free1}, we obtain a Rota-Baxter algebra $(\calh_{\geq 1}\qnonu, \qssha,I)$ where
\begin{eqnarray}
&&[\vec{s}]\qssha [\vec{t}]:= \eta(\eta^{-1}([\vec s])\ssha \eta^{-1}([\vec t])), \quad [\vec{s}],[\vec{t}]\in \calh_{\geq 1}\qnonu, \mlabel{eq:qssha}
\\
&& I:\calh_{\geq 1}\qnonu\to \calh_{\geq 1}\qnonu, I([\vec{s}]):= [\vec{s}+\vec{e}_1], \quad [\vec{s}]\in \calh_{\geq 1}\qnonu,
\mlabel{eq:Iop}
\end{eqnarray}
where $\vec{e}_1=(1,0,\cdots,0)$ is the first standard basis of $\ZZ^k$ if $k$ is the dimension of $\vec{s}$.
Then Theorem~\mref{thm:free1} can be rephrased as
\begin{theorem}
The Rota-Baxter algebra $(\calh_{\geq 1}\qnonu, \qssha,I)$ is the free commutative nonunitary Rota-Baxter algebra of weight 0 generated by $\eta(x_1)=[1]$.
\mlabel{thm:free1q}
\end{theorem}
We will call $\qssha$ the {\bf shuffle product on $\calh_{\geq 1}\qnonu$}. By the recursive definition of the shuffle product $\ssha$ on $\calh\shf\lone$, we obtain the recursive description of the shuffle product $\qssha$ on $\calh_{\geq 1}\qnonu$. For $\vec{s}=(s_1,\cdots,s_m), \vec{t}=(t_1,\cdots,t_n)$, we have (see also~\cite[Proposition 4.3]{GX2})
\begin{equation}
[\vec{s}]\qssha [\vec{t}] = \left\{\begin{array}{ll}
I([\vec{s}-\vec{e}_1]\qssha [\vec{t}]) + I([\vec{s}]\qssha [\vec{t}-\vec{e}_1]), & s_1,t_1>1, \\
{}[1, \vec{s}'\qssha \vec{t}]+ I([\vec{s}]\qssha [\vec{t}-\vec{e}_1]), & \vec{s}=[1,\vec{s}'], t_1>1, \\
I([\vec{s}-\vec{e}_1]\qssha [\vec{t}])+ [1,[\vec{s}]\qssha [\vec{t}']], & s_1>1, \vec{t}=[1,\vec{t}'],\\
{}[1,[\vec{s}']\qssha [\vec{t}]] + [1,[\vec{s}] \qssha [\vec{t}']], &\vec{s}=[1,\vec{s}'], \vec{t}=[1,\vec{t}'].
\end{array}
\right .
\mlabel{eq:shpr1}
\end{equation}

\subsection {Rota-Baxter algebra on $\calh_{\geq 0}\qnonu$}
\mlabel{ss:sh0}

The operator $I$ in Eq.~(\mref{eq:Iop}) extends to an operator
\begin{equation}
 I: \calh_{\ge 0}\qnonu \to
\calh_{\ge 0}\qnonu,\quad [\vec s]\mapsto [\vec s +\vec e_1].
\mlabel{eq:e1shift}
\end{equation}

\begin {theorem} The shuffle product $\qssha$ on $\calh _{\geq 1}\qnonu$ has a unique extension to a commutative associative product on $\calh _{\ge 0}\qnonu$, still denoted by $\qssha$, such that $\nvec{0}\qsshab \nvec{s}=[0,\vec{s}]$ and such that  $I$ is a Rota-Baxter operator on $\calh_{\ge 0}\qnonu$ of weight 0.
\mlabel{thm:shbax}
\end {theorem}

\begin{proof} We first prove the existence.
For $\nvec{s}=(s_1,\cdots,s_i)\in \ZZ_{\geq 0}^i$ and
$\nvec{t}=(t_1,\cdots,t_j)\in \ZZ_{\geq 0}^j$, we use induction on
$$c=s_1+\cdots+s_i+i+t_1+\cdots+t_j+j$$
to define $\nvec{s}\qsshab \nvec{t}$. Note that we have $c\geq 2$.

When $c=2$, then $i=j=1$ and $\nvec{s}=\nvec{t}=(0)$. Then define
$$ \nvec{s} \qsshab \nvec{t}= [0,0].$$
Suppose that $\nvec{s}\qsshab \nvec{t}$ have been defined for $c= n$.
Then for $\nvec{s}$ and $\nvec{t}$ with $c=n+1$, define
\begin{equation}
\nvec{s} \qsshab \nvec{t} = \left \{\begin{array}{ll} [0,\vec{s}'\qsshab
\vec{t}], & \nvec{s}=[0,\vec{s}'],
\\
 {}[0,\vec{s}\qsshab \vec{t}'], & \nvec{t}=[0,\vec{t}'],\\
I([\vec{s}-\vec{e}_1]\qsshab \nvec{t}) +I( \nvec{s}\qsshab
[\vec{t}-\vec{e}_1]), & {\rm otherwise}.
\end{array} \right .
\mlabel{eq:extsh}
\end{equation}
Then the terms on the right hand side are well-defined by the
induction hypothesis. We note that if $\nvec{s}=[0,\vec{s}']$ and
$\nvec{t}=[0,\vec{t}']$, then we have
$\nvec{s}\qsshab\nvec{t}=[0,0,\vec{s}'\qsshab \vec{t}'].$ So there
is no ambiguity in the above definition. It follows from
Eq.~(\mref{eq:shpr1}) that the restriction of the new
product $\qssha$ to $\calh_{\geq 1}\qnonu$ coincides with the product
$\qssha$ on $\calh_{\geq 1}\qnonu$.

Clearly $\qsshab$ is commutative. It is also clear that
Eq.~(\mref{eq:extsh}) is the only possible way to define $\qsshab$
satisfying the conditions in the theorem.

We next verify the associativity: for $\nvec{s}=[s_1,\cdots,s_i],
\nvec{t}=[t_1,\cdots,t_j]$ and $\nvec{u}=[u_1,\cdots,u_k]$,
\begin{equation} (\nvec{s} \qsshab \nvec{t}) \qsshab \nvec{u}
=\nvec{s} \qsshab (\nvec{t}\qsshab \nvec{u}). \mlabel{eq:shass}
\end{equation}
For this we use induction on
$$ d =s_1+\cdots+s_i+i+t_1+\cdots+t_j+j+u_1+\cdots+u_k+k.$$
Then $d\geq 3$. If $d=3$, then $i=j=k=0$ and
$\nvec{s}=\nvec{t}=\nvec{u}=[0]$. So both sides of
Eq.~(\mref{eq:shass}) is $[0,0,0]$. Suppose Eq.~(\mref{eq:shass})
has been verified for $d=n$ and take $\nvec{s},\nvec{t},\nvec{u}$
with $d=n+1$. If $\nvec{s}=[0,\vec{s}']$, then Eq.~(\mref{eq:shass})
means $$ [0,\big(\vec{s}'\qsshab \vec{t}\big) \qsshab \vec{u}] =
[0,\vec{s}'\qsshab \big(\vec{t}\qsshab \vec{u}\big)]$$ which follows
from the induction hypothesis. Similar arguments works if the first
component of $\nvec{t}$ or $\nvec{u}$ is 0.

It remains to consider the case when the first components of
$\nvec{s}, \nvec{t}$ and $\nvec{u}$ are all non-zero. Then by
Eq.~(\mref{eq:extsh}),
\begin{eqnarray*}
(\nvec{s}\qsshab \nvec{t})\qsshab \nvec{u} &=&
\big( I((\vec{s}-\vec{e}_1)\qsshab \vec{t} )+I(\vec{s}\qsshab (\vec{t}-\vec{e}_1))\big) \qsshab \vec{u} \\
&=& I\big( ((\vec{s}-\vec{e}_1)\qsshab \vec{t} )\qsshab \vec{u} \big)
+ I\big( I((\vec{s}-\vec{e}_1)\qsshab\vec{t} )\qsshab (\vec{u}-\vec{e}_1)\big) \\
&&+ I\big( (\vec{s} \qsshab (\vec{t}-\vec{e}_1))\qsshab \vec{u}\big) + I
\big(I (\vec{s}\qsshab (\vec{t}-\vec{e}_1))\ssha(\vec{u}-\vec{e}) \big)
\end{eqnarray*}
Applying the induction hypothesis to the first term on the right
hand side and use Eq.~(\mref{eq:extsh}) again, we have
\begin{eqnarray*}
(\nvec{s}\qsshab \nvec{t})\qsshab \nvec{u} &=& I\big(
(\vec{s}-\vec{e}_1)\qsshab I((\vec{t}-\vec{e}_1)\qsshab \vec{u} )\big)
+ I \big( (\vec{s}-\vec{e}_1)\qsshab I(\vec{t}\qsshab(\vec{u}-\vec{e}_1))\big) \\
&& + I\big( I((\vec{s}-\vec{e}_1)\qsshab\vec{t} )\qsshab
(\vec{u}-\vec{e}_1)\big)
+ I\big( (\vec{s} \qsshab (\vec{t}-\vec{e}_1))\qsshab \vec{u}\big) \\
&& + I \big(I (\vec{s}\qsshab
(\vec{t}-\vec{e}_1))\qsshab(\vec{u}-\vec{e}_1) \big).
\end{eqnarray*}

By the same argument, we find
\begin{eqnarray*}
\nvec{s} \qsshab (\nvec{t}\qsshab \nvec{u}) &=& I\big(
(\vec{s}-\vec{e}_1)\qsshab I((\vec{t}-\vec{e}_1)\qsshab \vec{u})\big)
+ I \big( \vec{s}\qsshab((\vec{t}-\vec{e}_1)\qsshab \vec{u})\big) \\
&& + I\big( (\vec{s}-\vec{e}_1)\qsshab I(
\vec{t}\qsshab(\vec{u}-\vec{e}_1) )\big) + I \big(
I((\vec{s}-\vec{e}_1)\qsshab \vec{t})\qsshab (\vec{u}-\vec{e}_1)\big)
\\ &&
+ I\big ( I (\vec{s}\qsshab (\vec{t}-\vec{e}_1))\qsshab
(\vec{u}-\vec{e}_1)\big).
\end{eqnarray*}
This agrees term-wise with the above sum for $(\nvec{s}\qsshab
\nvec{t})\qsshab \nvec{u}$ with another use of the induction
hypothesis.
\end{proof}

\subsection{Free Rota-Baxter algebra on $\calh_{\ge 0}\qnonu$}
\mlabel{ss:free0}
We now show that the Rota-Baxter algebra $\calh_{\ge 0}\qnonu$ obtained in Theorem~\mref{thm:shbax} is in fact free.

\begin{theorem}
The Rota-Baxter algebra $(\calh_{\ge 0}\qnonu,I)$ is the free commutative nonunitary Rota-Baxter algebra of weight 0 generated by $\lrb{0}$.
\mlabel{thm:hfree}
\end{theorem}
\begin{proof}
Instead of checking that the Rota-Baxter algebra $(\calh_{\ge 0}\qnonu,I)$ satisfies the desired universal property, we will show that this Rota-Baxter algebra is isomorphic to the free commutative nonunitary Rota-Baxter algebra $\sha(x\bfk[x])^0$ in Theorem~\mref{thm:free}.

\begin{lemma}
The nonunitary Rota-Baxter algebra $(\calh_{\ge 0}\qnonu,I)$ is generated by $\lrb{x}$.
\mlabel{lem:gen}
\end{lemma}
\begin{proof}
Let $R$ be the nonunitary Rota-Baxter sub-algebra of $\calh_{\ge 0}\qnonu$ generated by $\lrb{0}$. We just need to show that all the basis elements $\nvec{s}=(s_1,\cdots,s_k)\in \ZZ_{\geq 0}^k$ can be obtained by repeated applications of multiplication and the Rota-Baxter operator $I$ to $\lrb{0}$. But this follows since
$$ (s_1,\cdots,s_k)=I^{s_1}(\lrb{0}\qsshab I^{s_2}([0]\qsshab I^{s_3}\cdots I^{s_k}(\lrb{0})\cdots))$$
whose proof follows from a simple induction.
\end{proof}

Since $\sha (x\bfk[x])^0$ is the free commutative nonunitary Rota-Baxter algebra generated by $x$, by its universal property, there is a unique homomorphism of commutative nonunitary Rota-Baxter algebras
$$ \phi: \sha(x\bfk[x])^0 \to \calh_{\ge 0}\qnonu$$
such that $\phi(x)=\lrb{0}.$ By Lemma~\mref{lem:gen}, $\phi$ is surjective. By an inductive argument, we see that
$$ \phi(x^n)=\lrb{\underbrace{0,\cdots,0}_{n{\rm -times}}}$$
and in general
$$ \phi(x^{n_0}\ot x^{n_1}\ot \cdots \ot x^{n_\ell})
= \lrb{\underbrace{0,\cdots,0}_{n_0{\rm -times}},1,
\underbrace{0,\cdots,0}_{(n_1-1){\rm -times}},1,\cdots,1,
\underbrace{0,\cdots,0}_{(n_\ell-1){\rm -times}}}
$$
with the convention that if $n_i=0$, then
$(1,\underbrace{0,\cdots,0}_{(n_i-1){\rm -times}},1)=2$,
and if $n_i=n_{i+1}=0$, then
$(1,\underbrace{0,\cdots,0}_{(n_i-1){\rm -times}},1,\underbrace{0,\cdots,0}_{(n_{i+1}-1){\rm -times}},1)=3$, {\em etc.} Note that $n_\ell\geq 1$ by definition.
Now it is clear that $\phi$ sends two distinct basis elements of $\sha(x\bfk[x])^0$ to distinct basis elements of $\calh_{\ge 0}\qnonu$.
Therefore $\phi$ is injective. This completes the proof.
\end{proof}

\section{Extended shuffle relation and double shuffle relations from free Rota-Baxter
algebras} \mlabel{sec:app} We apply the freeness property of the
shuffle algebras $\calh\shf\lone\nonu$ and $\calh_{\geq 0}\qnonu$ as
nonunitary Rota-Baxter algebras to study multiple polylogarithms and
MZVs. We first generate the shuffle relation of multiple
polylogarithms in Section~\mref{ss:poly}. We then generate the
extended shuffle relation of MZVs in Section~\mref{ss:sh1}. In
Section~\mref{ss:dsh1}, we derive the extended double shuffle
relations of Ihara, Kaneko and Zagier~\mcite{IKZ}.

\subsection {Shuffle relations of multiple polylogarithms}
\mlabel{ss:poly} We first construct a Rota-Baxter algebra for the
study of multiple polylogarithms and MZVs.
Let $\CC\{\{\vep, \vep^{-1}\}$ be the algebra of convergent Laurent
series, regarded as a subalgebra of the algebra of (germs of)
complex valued functions meromorphic in a neighborhood of $\vep=0$.
We take $\ln (-\vep)$ to be component which is analytic on $\CC\backslash [0,\infty)$.

By \cite[Lemma 3.2]{GZ}, we have
\begin{lemma}
The function $\ln (-\vep) $ is transcendental over $\CC\{\{\vep,\vep^{-1}\}$ and hence over $\CC\{\{\vep\}\}$.
\mlabel{lem:vep}
\end{lemma}

\begin{defn}
{\
Let
$C^{Log}_S(-\infty,0)$ denote the subset of $\CC\{\{\vep,\vep^{-1}\}[\ln(-\vep)]$ as functions on $(-\infty,0)$ consisting of $f$ such that, for every $n\in \NN$, we have $\lim\limits_{\vep\to -\infty} \vep ^nf(\vep)= 0$.
}
\mlabel{de:logc}
\end{defn}

\begin {lemma} The complex vector space $C^{Log}_S(-\infty,0)$ is closed under function multiplication. The operator
\begin{equation}
J: C^{Log}_S(-\infty,0) \to C^{Log}_S(-\infty,0), \quad f\mapsto \int _{-\infty}^\vep f(t ) dt, f\in C^{Log}_S(-\infty,0),
\mlabel{eq:logint}
\end{equation}
is a Rota-Baxter operator of weight 0.
\mlabel{lem:jop}
\end {lemma}
\begin {proof}
Let $Y$ be the set of functions on $(-\infty,0)$ such that, for every $n\in \NN$, we have $\lim\limits_{\vep\to -\infty} \vep ^nf(\vep)= 0$. Then $C^{Log}_S(-\infty,0)=\CC\{\{\vep,\vep^{-1}\}[\ln(-\vep)]\cap Y$. $Y$ is obviously closed under function multiplication. Since $\CC[\ln(-\vep)]\{\{\vep,\vep^{-1}\}$ is also closed under function multiplication, so is $C^{Log}_S(-\infty,0)$.

By Lemma~3.2 of~\mcite{GZ}, the set $\CC\{\{\vep,\vep^{-1}\}[\ln(-\vep)]$ is closed under indefinite integral. The
condition of a function $f(\vep)$ in $C^{Log}_S(-\infty,0)$ at $-\infty$ ensures that an indefinite integral of $f$ can be evaluated at $-\infty$. Thus $\int_{-\infty}^\vep f(t)dt$ is well-defined and is still in $C^{Log}_S(-\infty,0)$. The operator is a Rota-Baxter operator of weight 0 because of
the integration by parts formula of integration operators. See
Example~\mref{ex:int}.
\end {proof}

We consider a special element
\begin{equation}
\frac{e^{\vep}}{1-e^{\vep}}= -\frac{1}{\vep} +\sum_{i=0}^\infty \zeta(-i) \frac{\vep^i}{i!}.
\mlabel{eq:gen0}
\end{equation}
It is in $C^{Log}_S(-\infty,0)$ since
$\lim\limits_{\vep\to -\infty}\vep^n e^{\vep}=0$ for $n\in \NN$.
Our interest in this element comes from the expansion
$$ \frac{e^{\vep}}{1-e^{\vep}} = \sum_{n=1}^\infty e^{n\vep}$$ which can be viewed as the regularization of the formal special value
$\zeta(0):=\sum\limits_{n=1}^\infty \frac{1}{n^0}$
of $\zeta(s)$ at $s=0$. See Section~\mref{sss:ren} for further details.

By Theorem \mref {thm:hfree} we have

\begin{prop} There is a unique nonunitary Rota-Baxter algebra homomorphism
$$\phi : \calh_{\ge 0}\qnonu\to C^{Log}_S(-\infty,0)$$
such that
$$\phi ([0])=\frac {e^ \vep }{1-e^\vep}.
$$
\end{prop}
In order to relate $\phi$ to multiple polylogarithms, we will need another property of Rota-Baxter algebras.
\begin{lemma}
Let $(\sha(x\bfk[x])^0,P_x)$ be the free commutative nonunitary Rota-Baxter algebra in Eq.~(\mref{eq:freenrb}). Let $(R,P)$ be a commutative nonunitary Rota-Baxter algebra. Let $f:\sha(x\bfk[x])^0\to R$ be a linear map such that
\begin{enumerate}
\item
$f(xy)=f(x)f(y), \forall y\in \sha(x\bfk[x])$;
\mlabel{it:mult}
\item
$f(P_x(y))=P(f(y)), \forall y\in \sha(x\bfk[x]).$
\mlabel{it:op}
\end{enumerate}
Then $f$ is a homomorphism of nonunitary Rota-Baxter algebras.
\mlabel{lem:indrb}
\end{lemma}

\begin{proof}
Let the linear map $f:\sha(x\bfk[x])^0 \to R$ with the properties in the lemma be given. By the universal property of $\sha(x\bfk[x])^0$ as the free commutative nonunitary Rota-Baxter algebra generated by $x$, there is a unique nonunitary Rota-Baxter algebra $g:\sha(x\bfk[x])^0 \to R$ such that $g(x)=f(x)$. We just need to show that $g=f$.

Since $\sha(x\bfk[x])^0$ is additively spanned by the pure tensors
$\calx:=x^{n_0}\ot x^{n_1}\ot \cdots \ot x^{n_k}, n_i\geq 0, 0\leq i\leq k, n_k\geq 1, k\geq 0$, we just need to show that $f$ and $g$ agree on these pure tensors. We prove this claim by induction on $k\geq 0$. When $k=0$, we have $\calx=x^{n_0}$. By condition~(\mref{it:mult}) and an induction on $n_0\geq 1$, we obtain
\begin{equation}
f(x^{n_0}y)=f(x)^{n_0}f(y), \quad \forall y\in \sha(x\bfk[x])^0.
 \mlabel{eq:mult2}
\end{equation}
In particular we have $f(x^{n_0})=f(x)^{n_0}=g(x)^{n_0}$. Assume the claim has been proved for $k=r\geq 0$ and consider $x^{n_0}\ot x^{n_1}\ot \cdots \ot x^{n_{r+1}}$. Then by Eq.~(\mref{eq:mult2}) and the induction hypothesis, we obtain
\begin{eqnarray*}
f(x^{n_0} \ot x^{n_1}\ot \cdots \ot x^{n_{r+1}})&=&
f(x^{n_0} P_x(x^{n_1}\ot \cdots \ot x^{n_{r+1}}))\\
&=& f(x)^{n_0} f(P_x(x^{n_1}\ot \cdots \ot x^{n_{r+1}})) \\
&=& f(x)^{n_0} P(f(x^{n_1}\ot \cdots \ot x^{n_{r+1}}))\\
&=& f(x)^{n_0} P(g(x^{n_1}\ot \cdots \ot x^{n_{r+1}})) \\
&=& g(x)^{n_0} g(P_x(x^{n_1}\ot \cdots \ot x^{n_{r+1}}))\\
&=& g(x^{n_0} \ot x^{n_1} \ot \cdots \ot x^{n_{r+1}}).
\end{eqnarray*}
This completes the induction.
\end{proof}

Now for $\vec{s}=(s_1,\cdots,s_k)\in \ZZ^k$, consider the {\bf polylogarithm}
\begin{equation}
Li_{\vec s}(z):= \sum_{n_1>\cdots>n_k\geq 1}
\frac{z^{n_1}}{n_1^{s_1}\cdots n_k^{s_k}}, \mlabel{eq:polylog}
\end{equation}
which is convergent for $z\in \CC$ with $|z| <1$.

\begin{theorem}
\begin{enumerate}
\item For $\vec{s}=(s_1,\cdots,s_k)\in \ZZ_{\geq 0}^k$, the function $Li_{\vec{s}}(e^{\vep})$ is in $C^{Log}_S(-\infty,0)$.
\mlabel{it:polylog}
\item
The linear map
\begin{equation}
\frakL: \calh_{\geq 0}\qnonu \la C^{Log}_S(-\infty,0),
\quad \vec{s}\mapsto Li_{\vec{s}}(e^{\vep})
\mlabel{eq:polydef}
\end{equation}
defines a Rota-Baxter algebra homomorphism.
\mlabel{it:poly}
\item
For any $\vec{s}\in \calh_{\geq 0}\qnonu$ we have
$Li_{\vec{s}}(e^{\vep}) = \phi([\vec{s}])(\vep).$
\mlabel{it:polyphi}
\end{enumerate}
\mlabel{thm:polyphi}
\end{theorem}
\begin {proof}
(\mref{it:polylog}).
We prove by induction on $k\geq 1$ with the help of the following two properties.
\begin{equation}
J(Li_{\vec{s}}(e^{\vep}))=\sum_{n_1>\cdots>n_k\geq 1}
\frac{\int _{-\infty}^\vep e^{n_1 t}d t}{n_1^{s_1}\cdots
n_k^{s_k}}= \sum_{n_1>\cdots>n_k\geq 1}
\frac{e^{n_1\vep}}{n_1^{s_1+1}\cdots n_k^{s_k}}=Li_{\vec{s}+\vec{e}_1}(e^{\vep}),
\mlabel{eq:polyJ}
\end{equation}

\begin{equation}
Li_{[0,\vec s]}(\vep)=\sum_{n_0>n_1>\cdots>n_k\geq 1}
\frac{e^{n_0 \vep}}{n_1^{s_1}\cdots n_k^{s_k}} =\frac {e^ \vep
}{1-e^\vep}\sum_{n_1>\cdots>n_k\geq 1} \frac{e^{n_1
\vep}}{n_1^{s_1}\cdots n_k^{s_k}}=
Li_{[0]}(e^{\vep}) Li_{\vec{s}}(e^{\vep}).
\mlabel{eq:poly1}
\end{equation}
When $k=1$, we have $Li_{(s_1)}(e^{\vep})=J^{s_1}(Li_{(0)}(e^{\vep}))$ which is in $C^{Log}_S(-\infty,0)$ by $Li_{(0)}(e^{\vep})\in C^{Log}_S(-\infty,0)$ and Eq.~(\mref{eq:polyJ}).
Suppose the statement has been proved for $k=r\geq 1$ and consider $\vec{s}=(s_1,\cdots,s_{r+1})$ in $\ZZ_{\geq 0}^{r+1}$. Then we have
$$Li_{\vec{s}}(e^{\vep}) =J^{s_1}\big(Li_{(0,s_2,\cdots,s_{r+1})}(e^{\vep})\big)
=J^{s_1}\big(Li_{[0]}(e^{\vep}) Li_{(s_2,\cdots,s_{r+1})}(e^{\vep})\big)$$
by Eq.~(\mref{eq:poly1}). This is in $C^{Log}_S(-\infty,0)$ by Lemma~\mref{lem:jop}, the induction hypothesis and Eq.~(\mref{eq:polyJ}). This completes the induction.

\smallskip

\noindent
(\mref{it:poly}). Let $I$ be the operator $[\vec s]\mapsto [\vec s+\vec
e_1]$ on $\calh_{\ge 0}\qnonu$ defined in Eq.~(\mref{eq:e1shift}) and $J$ be the Rota-Baxter operator on $C^{Log}_S(-\infty,0)$ defined in Eq.~(\mref{eq:logint}). Then we have
\begin{eqnarray*}
(\frakL\circ I)(\vec s)(\vep)&=&Li_{I(\vec s)}(\vep)\\
&=& \sum_{n_1>\cdots>n_k\geq 1}
\frac{e^{n_1\vep}}{n_1^{s_1+1}\cdots n_k^{s_k}} \\
&=&\sum_{n_1>\cdots>n_k\geq 1}
\frac{\int _{-\infty}^\vep e^{n_1\vep}d \vep}{n_1^{s_1}\cdots
n_k^{s_k}}\\
&=& J(\sum_{n_1>\cdots>n_k\geq 1}
\frac{e^{n_1\vep}}{n_1^{s_1}\cdots n_k^{s_k}}) \\
&=&(J\circ \frakL)(\vec s)(\vep).
\end{eqnarray*}
So $\frakL$ commutes with Rota-Baxter operators.

Also by Eq.~(\mref{eq:poly1}), we have
$$Li_{[0]\qsshab \vec s}(e^{\vep})=Li_{(0,\vec{s})}(e^{\vep})=Li_{[0]}(e^{\vep})Li_{\vec s}(e^{\vep}), \quad \forall \vec{s}\in \calh_{\geq 0}\qnonu.
$$
So $\frakL([0]\qsshab\vec s)=\frakL([0])\frakL([\vec{s}])$. Therefore by Lemma~\mref{lem:indrb}, $\frakL$ is a
homomorphism of nonunitary Rota-Baxter algebras.

\smallskip

\noindent
(\mref{it:polyphi}). By the universal property of the free commutative nonunitary Rota-Baxter algebra $\sha(x\bfk[x])^0$, the map $\frakL:\sha(x\bfk[x])^0\to C^{Log}_S(-\infty,0)$ is in fact the unique Rota-Baxter algebra homomorphism from $\sha(x\bfk[x])^0$ such that
$\frakL([0])= \frac{e^{\vep}}{1-e^{\vep}}$. Since $\phi$ also satisfies this property, we have $\frakL= \phi$. This is what we need.
\end{proof}

So by viewing $\calh_{\ge 0}\qnonu$ as a free nonunitary Rota-Baxter algebra
and choosing a suitable value for $[0]$, we obtained all the multiple polylogarithms as regularized MZVs with the shuffle product. By choosing other values for $[0]$, we can obtain other regularized MZVs with the shuffle product. In general, these regularized MZVs have
poles. For example $Li_0(\vep)=\frac {e^ \vep }{1-e^\vep}$ has an
order 1 pole at $\vep=0$. So in general, we can not take $\vep \to
0$.

\subsection{Extended shuffle relation}
\mlabel{ss:sh1} Having obtained regularized MZVs with shuffle product from the free Rota-Baxter algebra on $\calh_{\ge 0}\qnonu$, let us restrict our attention to $\calh_{\ge 1}\qnonu$.

The restriction of the Rota-Baxter algebra homomorphism $\phi:\calh_{\geq 0}\qnonu \to C^{Log}_S(-\infty,0)$ to $\calh_{\geq 1}\qnonu$ gives a Rota-Baxter algebra homomorphism
\begin{equation}
\phi: \calh_{\geq 1}\qnonu \to C^{Log}_S(-\infty,0).
\mlabel{eq:phi1}
\end{equation}
Since $\calh_{\geq 1}\qnonu$ is the free commutative nonunitary Rota-Baxter algebra generated by $[1]$ by Theorem~\mref{thm:free1q}, this $\phi$ is the unique Rota-Baxter algebra homomorphism such that
\begin{eqnarray*}
\phi(x_1)&=&\phi(I([0]))\\
&=&J(\phi([0]))\\
&=& J(\frac{e^{\vep}}{1-e^{\vep}}) \\
&=&-\ln(-\vep)+\sum _{i\ge 1} \zeta(-i+1)\frac {\vep ^i}{i!}.
\end{eqnarray*}

Since $\phi([1])$ is in $\CC\{\{\vep\}\}[\ln(-\vep)]$ which is closed under multiplication and taking antiderivatives, and $\phi(\calh_{\geq 1}\qnonu)$ is a Rota-Baxter algebra generated by $\phi([1])$, it follows that $\phi(\calh_{\geq 1}\qnonu)$ is contained in $\CC[[\vep]][\ln(-\vep)]$, the polynomial algebra over the formal Laurent series.
This can also be seen by viewing the images of $\phi=\frakL$ as multiple polylogarithms $Li_{\vec{s}}(e^{\vep})$.
Because $\ln
(-\vep) $ is transcendental over $\CC \{\{\vep\}\}$, we have the
embedding
$$ u: \CC\{\{\vep\}\}[\ln(-\vep)] \cong \CC\{\{\vep\}\}[T] \hookrightarrow \CC[[\vep]][T]\hookrightarrow \CC[T][[\vep]]$$
by sending $-\ln(-\vep)$ to $T$.
Pre-composing $\phi$ with $\eta: \calh\shf\lone\nonu\to \calh_{\geq 1}\qnonu$ in Eq.~(\mref{eq:shqsh}) and postcomposing $\phi$ with $u$ and then
the evaluation map $\vep\to 0$, we obtain an algebra homomorphism
\begin{equation}
 Z^{RB}: \calh\shf\lone\nonu \ola{\eta} \calh_{\geq 1}\qnonu \ola{\phi} \CC\{\{\vep\}\}[\lne] \ola{u} \CC[T][[\vep]] \ola{\vep\mapsto 0} \CC[T]
\mlabel{eq:rbmap}
\end{equation}
which we can extend to $\calh\shf\lone$ by unitarization.

We next compare $Z^{RB}$ with the extended shuffle
relation of MZVs~\mcite{IKZ}. We first recall some more notations.

As is well-known, an MZV has an integral representation~\mcite{LM}
\begin{equation}
\zeta(s_1,\cdots,s_k)=
 \int_0^1 \int_0^{t_1}\cdots \int_0^{t_{|\vec{s}|-1}} \frac{dt_1}{f_1(t_1)}
 \cdots \frac{dt_{|\vec{s}|}}{f_{|\vec{s}|}(t_{|\vec{s}|})}
 \mlabel{eq:intrep}
\end{equation}
Here $|\vec{s}|=s_1+\cdots +s_k$ and
$$f_j(t)=\left\{\begin{array}{ll} 1-t_j, & j= s_1,s_1+s_2,\cdots, s_1+\cdots +s_k,\\
t_j, & \text{otherwise}. \end{array} \right .
$$
Since the integral operator is the Rota-Baxter operator of weight
zero in Example~\mref{ex:int}, it is expected that the
multiplication of two MZVs is given by the shuffle product that
defines the product in a free commutative Rota-Baxter algebra of
weight 0. This is indeed the case. Let
$$
\mzvalg\qnonu: = \QQ \{ \zeta(s_1,\cdots,s_k)\ |\ s_i\geq 1, s_1\geq 2\}
\subseteq \RR
$$
be the $\QQ$-subspace of $\RR$ spanned by MZVs and let $$\mzvalg=\QQ
+ \mzvalg\qnonu \subseteq \RR.$$ Then the shuffle product of MZVs is encoded by the algebra homomorphism~\mcite{Ho1,IKZ}
$$
 \zeta\shf: \calh\shf\shzero \to \mzvalg,
 \quad x_0^{  s_1-1}  x_1  \cdots   x_0^{s_k-1}  x_1 \mapsto \zeta(s_1,\cdots,s_k), \quad 1\mapsto 1.
$$
Furthermore, note that
$\calh\shf\lone\cong \calh\shf\shzero[y]$ where $y$ is a polynomial
variable. A canonical choice of $y$ is $x_1$. Thus the algebraic
homomorphisms $\zeta^{\ssha}$ extends uniquely to an algebraic
homomorphism
\begin{equation}
Z^{\ssha}:\calh\shf\lone \to \mzvalg [T]
\mlabel{eq:zsh}
\end{equation}
sending $y$ to $T$~\mcite{IKZ}. This is the extended shuffle
relation.

\begin{theorem}

We have $Z^{RB}=Z^{\ssha}$. In particular, the
restriction of $Z^{RB}$ to $\calh\shf\lzero$ agrees with
$\zeta^{\ssha}$.
\mlabel{thm:rbmzv}
\end{theorem}

\begin{proof} For $\vec{s}=(s_1,\cdots,s_k)$ with
$s_1>1$ and $s_i\geq 1$, $1\leq i\leq k$, by Theorem~\mref{thm:polyphi}.(\mref{it:polyphi}) evaluated at $\vep=0$, we have
$Z^{RB}(\frakx_{\vec{s}})=\zeta^{\ssha}(\frakx_{\vec{s}})$. So
$Z^{RB}$ and $\zeta^{\ssha}$ agrees on $\calh\shf\lzero$. Then
the theorem follows since both $Z^{RB}$ and
$Z^{\ssha}$ are the unique extension of $\zeta^{\ssha}:
\calh\shf\lzero \to \CC$ by taking $z_1$ to $T$.
\end{proof}

\subsection {Extended double shuffle relations}
\mlabel{ss:dsh1} We have just derived the extended shuffle relation
$Z^{\ssha} $ of MZVs through the freeness of $\calh\shf\lone\nonu$.
In earlier papers~\mcite{GZ,GZ2,MP2} we have also studied the extended
stuffle (quasi-shuffle) relation $\zeta^*$ of MZVs by renormalization. By combing
these two together, we next derive the extended double shuffle
relations (EDS)~\mcite{IKZ}. To formulate the results, we first give
a summary of EDS and regularized MZVs.

\subsubsection{Extended double shuffle relations}
Since an MZV is defined as a nested sum in Eq.~(\mref{eq:mzv}) and the
summation operator is the Rota-Baxter operator of weight 1 in
Example~\mref{ex:sum}, the multiplication of two MZVs follow the quasi-shuffle product
(mixable shuffle product of weight 1) that defines the
multiplication in a free commutative Rota-Baxter algebra of weight 1.
More precisely, consider the semigroup
$$ G_{\geq 1}:=\{z_s:=[s]\ |\ s\in \ZZ_{\geq 1}\}$$
in Notation~\mref{no:mscase}. Then the usual quasi-shuffle algebra
for MZVs is $$\calh\qsh:=\calh\qsh_{\hspace{-.2cm}\QQ\, Z}=
\calh_{\geq 1}^\ast$$ which contains the subalgebra
$$
\calh\qsh\lzero:=\QQ \oplus \Big (\bigoplus_{s_i\geq 1, 1\leq i\leq k, s_1>1, k\geq 1} \QQ z_{s_1}
\cdots z_{s_k} \Big ).
$$
Then the stuffle (quasi-shuffle) product of MZVs is encoded by the algebra homomorphism~\mcite{Ho1,IKZ}
\begin{equation}
\zeta\qsh: \calh\qsh\lzero \to \mzvalg, \quad z_{s_1}\cdots z_{s_k}
\mapsto \zeta(s_1,\cdots,s_k), \quad 1\mapsto 1. \mlabel{eq:mzvst}
\end{equation}

The natural bijection of $\QQ$-vector spaces
$$
\shqs: \calh\shf\lone \to \calh\qsh, \quad x_0^{s_1-1}  x_1  \cdots
x_0^{s_k-1}  x_1
 \leftrightarrow z_{s_1, \cdots, s_k}, \quad 1\leftrightarrow 1.
$$
restricts to a bijection of vector spaces $\shqs:
\calh\shf\shzero \to \calh\qsh\lzero. $ Then the fact that the product of two MZVs can be expressed in two ways is encoded by the commutative diagram
\begin{equation}
\xymatrix{ \calh\qsh\lzero  \ar_{\zeta\qsh}[rd]&&
    \calh\shf\shzero \ar_{\shqs}[ll] \ar^{\zeta\shf}[ld] \\
& \mzvalg & } \mlabel{eq:diag1}
\end{equation}
Defining a product $\qssha$ on $\calh\qsh\lzero$ from $\ssha$ through $\eta$, the {\bf double shuffle relation} is the set
$$
\{ w_1\qssha w_2 - w_1 \ast w_2\ |\ w_1,w_2\in\calh\qsh\lzero\}.$$

Since
$\calh\qsh \cong \calh\qsh\lzero[y]$ where $y$ is a polynomial
variable, the algebraic homomorphism
$\zeta^\ast$ extends uniquely to an algebraic homomorphism~\mcite{IKZ}
\begin{equation}
Z^\ast: \calh\qsh \to \mzvalg[T]
\mlabel{eq:zqsh}
\end{equation}
sending $y$ to $T$. Define a
function $A(u)$ and its Taylor series expansion by
\begin{equation}
 A(u)=\exp \big( \sum_{n=2}^\infty \frac{(-1)^n}{n} \zeta(n) u^n\big)=\sum_{k=0}^\infty \gamma_k u^k, \gamma_k\in \RR
\mlabel{eq:A}
\end{equation}
and define a map $\rho:\RR[T] \to \RR[T]$ by
\begin{equation}
\rho(e^{Tu})=A(u)e^{Tu}.
\mlabel{eq:rho}
\end{equation}
Then the commutative diagram in Eq.~(\mref{eq:diag1}) extends to the
commutative diagram
\begin{equation}
\xymatrix{ \calh\qsh  \ar_{Z\qsh}[d]&&
    \calh\shf\lone \ar_{\shqs[y]}[ll] \ar^{Z\shf}[d] \\
\mzvalg[T] & & \mzvalg[T] \ar_{\rho}[ll] } \mlabel{eq:diag2}
\end{equation}
where $\eta[y]$ is extended from $\eta$ by sending $y$ to $y$.

The {\bf extended double shuffle relation}~\mcite{IKZ,Ra,Zud} is
\begin{equation}
\{ w_1\qssha w_2 - w_1 \ast w_2,\ z_1 \qssha w_2 - z_1 \ast w_2\ |\
w_1,w_2\in\calh\qsh\lzero\}. \mlabel{eq:eds}
\end{equation}
\begin{theorem} {\bf (\mcite{Ho1,IKZ,Ra})}
Let $I_\edsalg$ be the ideal of $\calh\qsh\lzero$ generated by the
extended double shuffle relation in Eq.~(\mref{eq:eds}). Then
$I_\edsalg$ is in the kernel of $\zeta\qsh$. \mlabel{thm:eds}
\end{theorem}
It is conjectured that $I_\edsalg$ is in fact the kernel of
$\zeta\qsh$. A consequence of this conjecture is the irrationality
of $\zeta(2n+1), n\geq 1$~\mcite{An}.

\subsubsection{Renormalized MZVs}
\mlabel{sss:ren} To extend the double shuffle relations to MZVs with
non-positive arguments, we have to make sense of the divergent sums
defining these MZVs. We give a summary of the renormalization
approach and refer the reader to other
references~\mcite{Gugn,GPXZ,GZ,GZ2} for
details.

Consider the abelian semigroup
\begin{equation}
\frakM= \{{{\wvec{s}{r}}}\ \big|\ (s,r)\in \ZZ \times \RR_{>0}\}
\end{equation}
with the multiplication
$$ {\wvec{s}{r}}\cdot {\wvec{s'}{r'}}={\wvec{s+s'}{r+r'}}.$$
With the notation in Section~\mref{ss:msh}, we define the quasi-shuffle algebra
algebra
$$\calh_{\frakM}\qsh:=\sh_{\CC,1}(\CC \frakM)$$
with the quasi-shuffle product $*$. For $w_i=\wvec{s_i}{r_i}\in
\frakM,\ i=1,\cdots,k$, we use the notations
$$ \vec{w}=(w_1, \dots,w_k)
=\wvec{s_1, \dots,s_k}{r_1, \dots,r_k}=\wvec{\vec s}{\vec r},\ {\rm
where\ } \vec s=(s_1, \dots,s_k), \vec r=(r_1, \dots,r_k).$$
For $\vec{w}=\wvec{\vec s}{\vec r}\in \frakM^k$ and $\vep\in \CC$
with ${\rm Re}(\vep)<0$, define the {\bf directional regularized
MZV}:
\begin{equation}
Z(\wvec{\vs}{\vr};\vep)=\sum_{n_1>\cdots>n_k>0}
\frac{e^{n_1\,r_1\vep} \cdots
    e^{n_k\,r_k\vep}}{n_1^{s_1}\cdots n_k^{s_k}}.
\label{eq:reggmzv}
\end{equation}
It converges for any $\wvec{\vs}{\vr}$ and is regarded as the
regularization of the {\bf formal MZV}
\begin{equation}
\zeta (\vs)= \sum_{n_1>\cdots>n_k>0} \frac{1}{n_1^{s_1} \cdots
    n_k^{s_k}}
\label{eq:formgmzv}
\end{equation}
which converges only when $s_i>0$ and $s_1>1$. Notice that
$$Z(\wvec{\vs}{\vec e_1};\vep)=Li_{\vec s}(e^\vep).
$$

This regularization defines an algebra homomorphism~\mcite{GZ}:
\begin{equation}
\uni {Z}: \calh_\frakM \to \CC[T][[\vep,\vep^{-1}]. \label{eq:zmap}
\end{equation}

\subsubsection{Double shuffle of regularized MZVs}

We now derive the extended double shuffle relation from regularized
MZVs. We start with some preparational lemmas.
For $\ell\geq 1$, denote $\{1\}^\ell = \underbrace{1,\cdots,1}_{\ell-\text{terms}}$.

\begin {lemma} For $\vec s\in \ZZ ^k_{>0}$ with $s_1>1$,
$$ Li_{(\{1\}^{\ell},\vec{s})}(e^\vep)=\sum_{m_1>m_2>\cdots >m_\ell>n_1>\cdots>n_k\geq 1} \frac{e^{m_1\vep}}{m_1\cdots m_\ell n_1^{s_1}\cdots n_k^{s_k}} $$
is of order $(\lne )^\ell$, i.e. it is $a_\ell  (\lne )^\ell+a_{\ell-1}(\lne
)^{\ell-1}+\cdots +a_0+o(\vep)$, where $a_i$'s are constants.
\mlabel{lem:order1}
\end{lemma}

\begin {proof} We prove this lemma by induction on $\ell$. When $\ell=0$ it is obvious because of the convergency of $\zeta (\vec s)$ which gives $Li_{\vec{s}}(e^\vep)=\zeta (\vec s)+o(\vep)$. Assume that the lemma has been proved for the case when $\ell=a\geq 0$. The shuffle relation in $\calh\shf\lone$
$$ x_1\ssha x_1^a\frakx_{\vec{s}} = (a+1)x_1^{a+1}\frakx_{\vec{s}} + x_1^ax_0(x_1\ssha \frakx_{\vec{s}-\vec{e}_1})$$
translates to the relation in $(\calh_{\geq 1}\qnonu,\qssha)$:
$$ [1]\qssha [\{1\}^a,\vec{s}] = (a+1)[\{1\}^{a+1},\vec{s}] + [\{1\}^a, ([1]\qssha (\vec{s}-\vec{e}_1))+\vec{e}_1].$$
Then we have
$$
Li_{[\{1\}^{a+1},\vec{s}]}(e^\vep)= \frac{1}{a+1} Li_{[1] \qssha
[\{1\}^a,\vec{s}]}(e^\vep) - \frac{1}{a+1} Li_{[\{1\}^a, ([1]\qssha (\vec{s}-\vec{e}_1))+\vec{e}_1]}(e^\vep).$$ For the first
term on the right hand side, we have
$$Li_{[1] \qssha
[\{1\}^a,\vec{s}]}(e^\vep)=Li_{[1]}(e^\vep) Li_{[\{1\}^a,\vec{s}]}(e^\vep)
$$
and
$$Li_{[1]}(e^\vep)=-\ln(-\vep)+\sum _{i\ge 1} \zeta(-i+1)\frac {\vep ^i}{i!}.
$$
Hence it is of order $(\ln(-\vep))^{a+1}$ by the induction
hypothesis. The second term on the right hand side is a linear
combination $\sum_{i=1}^r Li_{[\{1\}^a,\vec{s}'_i]}(e^\vep)$
where each $\vec{s}'_i$ has its first component greater than 1. Thus
by the induction hypothesis, the second term is of order
$(\ln(-\vep))^a$ or lower. This completes the induction.
\end{proof}

\begin {lemma} Let $\ell\geq  0, k\geq 1, \vec s\in \ZZ ^k_{>0}$ with $s_1>1$ and $\vec r \in \ZZ ^k_{\ge 0}$.
\begin{enumerate}
\item The nested sum
$$Z(\wvec{\{1\}^\ell,s_1,\cdots,s_k}{\{1\}^\ell,0,\cdots,0};\vep)
:=\sum_{m_1>m_2>\cdots >m_\ell>n_1>\cdots>n_k\geq 1} \frac{e^{m_1\vep}\cdots e^{m_\ell\vep}}{m_1\cdots m_\ell n_1^{s_1}\cdots n_k^{s_k}} $$
is of order $(\lne )^\ell$.
\mlabel{it:zsum1}
\item
For $1\leq j\leq k$, the sum
$\sum\limits_{m_1>\cdots >m_\ell>n_1>\cdots>n_k\geq 1} \frac{e^{m_1\vep}\cdots e^{m_\ell\vep}n_j}{m_1\cdots m_\ell n_1^{s_1}\cdots n_k^{s_k}}$
is of order at most $(\lne)^{\ell+k}$.
\mlabel{it:zsum2}
\item
We have the asymptomatic formula $$Z(\wvec{\{1\}^\ell,s_1,\cdots,s_k}{\{1\}^\ell,0,\cdots,0};\vep)
=Z(\wvec{\{1\}^\ell,s_1,\cdots,s_k}{\{1\}^\ell,s_1,\cdots,s_k}) +o(\vep).
$$
Hence $Z(\wvec{\{1\}^\ell,s_1,\cdots,s_k}{\{1\}^\ell,s_1,\cdots,s_k})$ is of order $(\lne)^\ell$.
\mlabel{it:zsum3}
\end{enumerate}
\mlabel{lem:zsum}
\end{lemma}
\begin {proof}
For $\ell=0$, all parts of the lemma are obvious because of the convergency of $\zeta (\vec s)$.
For $\ell\geq 1$, we prove the parts separately.

\noindent
(\mref{it:zsum1}). We prove by induction on $\ell\geq 1$.
The case when $\ell=1$ follows from Lemma~\mref{lem:order1}.
Assume that the case when $\ell=r\geq 1$ has been proved.  Consider
$$Z\left(\wvec{1}{1}* \wvec{\{1\}^\ell, \vec s}{\{1\}^\ell, \vec
0};\vep\right)=Z\left(\wvec{1}{1};\vep\right)Z\left(\wvec{\{1\}^\ell, \vec s}{\{1\}^\ell, \vec 0};\vep\right).
$$
By the quasi-shuffle product, the left hand side is of the form
$$
(\ell+1)Z\left(\wvec{\{1\}^{\ell+1},\vec{s}}{\{1\}^{\ell+1},\vec{0}};\vep\right)
+ \sum_{i} Z\left(\wvec{\vec{\{1\}^{c_i},\vec{s}_i}}{\vec{\{1\}^{c_i},\vec{t}_j}};\vep\right)
$$
 with $c_i\leq \ell$ and $\vec{s}_i$ having its first component
greater than $1$. Thus by the induction hypothesis, all the terms
except the first one are of order at most $(\lne )^{\ell}$.
Similarly, the right hand side is of order at most
$(\lne)^{\ell+1}$. Thus
$Z(\wvec{\{1\}^{\ell+1},\vec{s}}{\{1\}^{\ell+1},\vec{0}};\vep) $ is
order at most $(\lne)^{\ell+1}$. This proves the first part of the
lemma.

\smallskip

(\mref{it:zsum2}).
We prove by induction on $k\geq 1$. When $k=1$, we have $j=1$. Then by Item~(\mref{it:zsum1}), the sum
$$\sum_{m_1>\cdots>m_\ell>n_1\geq 1}\frac{e^{m_1\vep}\cdots e^{m_\ell\vep}n_1}{m_1\cdots m_\ell n_1^{s_1}}
=\sum_{m_1>\cdots>m_\ell>n_1\geq 1}\frac{e^{m_1\vep}\cdots e^{m_\ell\vep}}{m_1\cdots m_\ell n_1^{s_1-1}}$$
is of order at most $(\lne)^{\ell+1}$.

Assume the case of $k=r$ and consider
\begin{equation}
\sum_{m_1>\cdots>m_\ell>n_1>\cdots>n_{r+1}\geq 1} \frac{e^{m_1\vep}\cdots e^{m_\ell}n_j}{m_1 \cdots m_\ell
n_1^{s_1}\cdots n_k^{s_{r+1}}}=\sum_{m_1>\cdots >m_\ell>n_1>\cdots>n_{r+1}\geq 1} \frac{e^{m_1\vep}\cdots e^{m_\ell\vep}}{m_1\cdots m_\ell
n_1^{s_1}\cdots n_j^{s_j-1} \cdots n_k^{s_{r+1}}}, 1\leq j\leq r+1.
\mlabel{eq:order2}
\end{equation}
If $j=1$, then by Item~(\mref{it:zsum1}) again, the sum in Eq.~(\mref{eq:order2}) is of order at most $(\lne )^{r+k+1}$.
For $j>1$, we consider two cases. If $s_j>1$, then by Item~(\mref{it:zsum2}) again, the sum in Eq.~(\mref{eq:order2}) is of order $(\lne)^\ell$. If $s_j=1$, then $n_j$ does not appear in the summand, but still appears in the index set of the sum in Eq.~(\mref{eq:order2}).
Thus the sum is simplified to
$$
\sum_{m_1>\cdots>m_\ell>n_1>\cdots n_{j-1}>n_{j+1}\cdots >n_k\geq 1}
\frac{e^{m_1\vep}\cdots e^{m_\ell\vep}(n_{j-1}-n_{j+1}-1)}{m_1 \cdots m_\ell n_1^{s_1}\cdots n_{j-1}^{s_{j-1}}n_{j+1}^{s_{j+1}}
\cdots n_k^{s_{r+1}}}.
$$
Then by the induction hypothesis, the sum is of order at most $(\lne)^{r+k}$. This completes the induction.

\smallskip

\noindent
(\mref{it:zsum3}).
Note that, for any real number $x$,
$$e^x>1+x.$$
Thus in our case,
\begin{equation}
(n_1r_1+\cdots n_kr_k)(-\vep )>1-e^{(n_1r_1+\cdots n_kr_k)\vep }.
\mlabel{eq:ineq}
\end{equation}
Therefore
{\allowdisplaybreaks
\begin{eqnarray*}
\lefteqn{
Z\left(\wvec{\{1\}^\ell,s_1,\cdots,s_k}{\{1\}^\ell,0,\cdots,0};\vep\right)
-\sum_{m_1>\cdots
>m_\ell>n_1>\cdots>n_k\geq 1} \frac{e^{m_1\vep}\cdots
e^{m_\ell\vep}e^{n_1r_1\vep} \cdots e^{n_kr_k\vep}}{m_1\cdots m_\ell
n_1^{s_1}\cdots n_k^{s_k}}
}\\
&=&
\sum_{m_1>\cdots>m_\ell>n_1>\cdots>n_k\geq 1}
\frac{e^{m_1\vep}\cdots e^{m_\ell \vep}(1-e^{(n_1r_1+\cdots n_kr_k)\vep })}{m_1\cdots m_\ell
n_1^{s_1}\cdots n_k^{s_k}}\\
&<& \sum_{m_1>\cdots>m_\ell>n_1>\cdots>n_k\geq 1}
\frac{e^{m_1\vep}\cdots e^{m_\ell\vep}(n_1r_1+\cdots n_kr_k)(-\vep )}{m_1\cdots m_\ell n_1^{s_1}\cdots
n_k^{s_k}}\\
&=&
\sum_{i=1}^k r_i (-\vep) \left( \sum_{m_1>\cdots>m_\ell>n_1>\cdots>n_k\geq 1}
\frac{e^{m_1\vep}\cdots e^{m_\ell\vep}n_j}{m_1\cdots m_\ell n_1^{s_1}\cdots
n_k^{s_k}}\right).
\end{eqnarray*}
}
Thus Item~(\mref{it:zsum3}) follows from Item~(\mref{it:zsum2}).
\end {proof}

\begin {lemma} Let $\ell, k\geq 1$ and $\vec{s}=(s_1,\cdots,s_k)\in \ZZ_{\geq 1}^{k}$ with $s_1>1$ be given.
\begin{enumerate}
\item
There are $a_{i,j_i}\in \ZZ$, $\vec s_{i,j_i}\in \ZZ ^{k_{i,j_i}}$ with the first
component of $\vec{s}_{i,j_i}$ greater than $1$, where $k_{i,j_i}\geq 1,1\leq j_i\leq m_i, m_i\geq 1, 1\leq i\leq \ell$,  such that
\begin{equation}
[\{1\}^\ell, \vec s]=\sum_{1\leq j_i\leq m_i,0\leq i\leq\ell} a_{i,j_i}[\{1\}^i]*[\vec s_{i,j_i}].
\mlabel{eq:inda1}
\end{equation}
\mlabel{it:1inda}
\item
Let $a_{ij}\in \ZZ$ and $\vec s_{ij}\in \ZZ ^{k_{ij}}$ with $k_{ij}\geq 1$ be as given in Item~(\mref{it:1inda}).
Then for any $\vec{p}=(p_1,\cdots,p_\ell)\in \ZZ_{\geq 0}^\ell$ and $\vec{r}\in \ZZ_{\geq 0}^k$, there are $\vec{r}_{i,j_i}\in \ZZ_{\geq 0}^{k_{i,j_i}}, 1\leq j_i\leq m_i, m_i\geq 1, 1\leq i\leq \ell$, such that
\begin{equation}
\wvec{\{1\}^\ell, \vec s}{\vec{p},\vec{r}}=\sum_{1\leq j_i\leq m_i, 0\leq i\leq\ell} a_{i,j_i} \wvec{\{1\}^i}{p_1,\cdots,p_i}*\wvec{\vec s_{i,j_i}}{r_{i,j_i}}.
\mlabel{eq:indb1}
\end{equation}
\mlabel{it:1indb}
\end{enumerate}
\mlabel{lem:1ind}
\end{lemma}
\begin{proof} We prove this lemma by induction. For $\ell =1$,
by the quasi-shuffle relation, we have
$$[1]*[\vec s]=[1,\vec s]+\sum_{1\leq j_1 \leq m_1} c_{1,j_1}[\vec s_{1,j_1}],
$$
where $c_{1,j_1}\in \{0,1\}$ and $\vec{s}_{1,j_1}\in \ZZ_{\geq 1}^{k_{1,j_1}}$ have the first component greater than $1$.
So
$$[1,\vec s]=[1]*[\vec s]-\sum_{1\leq j_1 \leq m_1}c_{1,j_1}[\vec s_{1,j_1}],
$$
giving us the coefficients $a_{0,1}=1, a_{1,i_1}=-c_{1,i_1}$ and proving Item~(\mref{it:1inda}) when $\ell=1$.

Further for any $\vec{p}=(p_1)\in \ZZ_{\geq 0}^1$ and $\vec{r}\in \ZZ_{\geq 0}^k$, note that the quasi-shuffle product in $\calh_\frakM$ has the same effect on the first row and the second row of the basis elements $\wvec{\vec{s}}{\vec{r}}$. Thus we have
$$\wvec{1}{p_1}*\wvec{\vec s}{\vec{r}}=\wvec{1,\vec s}{p_1,\vec{r}}+\sum_{1\leq j_1 \leq m_1}c_{1,j_1}\wvec{\vec s_{1,j_1}}{\vec{r}_{1,j_1}},
$$
for the same $c_{1,j_1}$ in the last part of the proof. Here $\vec{r}_{1,j_1}\in \ZZ_{\geq 0}^{k_{1,j_1}}$. Thus for the same coefficients $a_{i,j_i}, 0\leq i\leq 1,$ in the last part of the proof, we have
$$
\wvec{1,\vec s}{p_1,\vec{r}}= a_{0,1}\wvec{1}{p_1}*\wvec{\vec s}{\vec{r}}+\sum_{1\leq j_1 \leq m_1}a_{1,j_1}\wvec{\vec s_{1,j_1}}{\vec{r}_{1,j_1}},
$$
proving Item~(\mref{it:1indb}) when $\ell=1$.

Now assume that the lemma is proved for $1, \cdots, \ell$.  Then by the quasi-shuffle product, we have
$$[\{1\}^{\ell +1}]*[\vec s]=[\{1\}^{\ell +1},\vec s]+\sum _{1\leq j_i\leq m_i, 0\leq i\leq \ell }c_{i,j_i}[\{1\}^i,\vec s_{i,j_i}],
$$
where $c_{i,j_i}\in \{0,1\}$ and $s_{i,j_i}\in \ZZ_{\geq 0}^{k_{i,j_i}}$.
Further for any $\vec{p}\in \ZZ_{\geq 0}^{\ell+1}$ and $\vec{r}\in \ZZ_{\geq 0}^k$, we also have
$$\wvec{\{1\}^{\ell +1}}{\ \vec{p}\ }*\wvec{\vec s}{\vec{r}}=\wvec{\{1\}^{\ell +1},\vec s}{\ \vec{p},\ \vec{r}}+\sum _{1\leq j_i\leq m_i, 0\leq i\leq \ell }c_{i,j_i}\wvec{\quad \{1\}^i,\quad \vec s_{i,j_i}}{p_1,\cdots,p_i,\vec{r}_{i,j_i}}
$$
for the same $c_{i,j_i}$ and some $\vec{r}_{i,j_i}\in\ZZ_{\geq 0}^{k_{i,j_i}}$.
Thus
\begin{equation}
[\{1\}^{\ell +1},\vec s]=[\{1\}^{\ell +1}]*[\vec s]-\sum _{0\leq i\leq \ell, j\geq 0}c_{i,j_i} [\{1\}^i,\vec s_{i,j_i}]
\mlabel{eq:inda2}
\end{equation}
and
\begin{equation}
\wvec{\{1\}^{\ell +1},\vec s}{\vec{p},\vec{r}}=\wvec{\{1\}^{\ell +1}}{\vec{p}}*\wvec{\vec s}{\vec{r}}-\sum _{1\leq j_i\leq m_i, 0\leq i\leq \ell }c_{i,j_i}\wvec{\quad \{1\}^i,\quad \vec s_{i,j_i}}{p_1,\cdots,p_i,\vec{r}_{i,j_i}}.
\mlabel{eq:indb2}
\end{equation}
By the induction hypothesis, the lemma applies the terms in the sums of the last two equations and gives expressions in Eqs.~(\mref{eq:inda1}) and (\mref{eq:indb1}). In particular, for each term in the sum in Eq.~(\mref{eq:inda2}), the coefficients $a_{i,j_i}$ in Eq.~(\mref{eq:inda1}) are the same as the coefficients in Eq.~(\mref{eq:indb1}) for the corresponding term in the sum in Eq.~(\mref{eq:indb2}).
Thus the lemma is proved for $\ell+1$, completing the induction.
\end{proof}

By Lemma~\mref{lem:1ind}, we have, for $\ell\geq 1$,
$$\wvec{\{1\}^\ell,\vec s}{\vec{e}_1,
\vec 0}=\sum_{1\leq j_i\leq m_i, 0\leq i\leq\ell} a_{i,j_i} \wvec{\{1\}^i}{\vec{e}_1}*\wvec{\vec
s_{i,j_i}}{\vec r_{i,j_i}}
$$
where $\vec{e}_1$ on the left (resp. right) hand side is the first unit vector of dimension $\ell$ (resp. $i$), and
$$\wvec{\{1\}^\ell,\vec s}{\{1\}^\ell,
\vec s}=\sum_{1\leq j_i\leq m_i, 0\leq i\leq\ell} a_{i,j_i} \wvec{\{1\}^i}{\{1\}^i}*\wvec{\vec
s_{i,j_i}}{\vec r\,'_{i,j_i}}
$$
for the $a_{i,j_i}\in \ZZ$ in Lemma~\mref{lem:1ind} and some
$\vec r_{i,j_i}, \vec r\,'_{i,j_i}\in
\ZZ ^{k_{i,j_i}}_{\ge 0}$.
Therefore, we have
\begin {equation}
 Z(\wvec{\{1\}^\ell,\vec s}{\vec{e}_1,
\vec 0};\vep)=\hspace{-.4cm}\sum_{1\leq j_i\leq m_i, 0\leq i\leq\ell} \hspace{-.4cm}a_{i,j_i} Z(\wvec{\{1\}^i}{\vec{e}_1}*\wvec{\vec s_{i,j_i}}{\vec r_{i,j_i}};\vep)
=\hspace{-.4cm}\sum_{1\leq j_i\leq m_i, 0\leq i\leq\ell} \hspace{-.4cm} a_{ij} Z(\wvec{\{1\}^i}{\vec{e}_1};\vep)Z(\wvec{\vec s_{i,j_i}}{\vec r_{i,j_i}};\vep) \mlabel
{eq:1expstu}
\end {equation}
and
\begin {equation}
 Z(\wvec{\{1\}^\ell,\vec s}{\{1\}^\ell,
\vec s};\vep)=\hspace{-.4cm}\sum_{1\leq j_i\leq m_i, 0\leq i\leq\ell}\hspace{-.4cm} a_{i,j_i} Z(\wvec{\{1\}^i}{\{1\}^i}*\wvec{\vec s_{i,j_i}}{\vec r\,'_{i,j_i}};\vep) =\hspace{-.4cm}\sum_{1\leq j_i\leq m_i, 0\leq i\leq\ell} \hspace{-.4cm} a_{i,j_i} Z(\wvec{\{1\}^i}{\{1\}^i}) Z(\wvec{\vec s_{i,j_i}}{\vec r\,'_{i,j_i}};\vep). \mlabel{eq:1expshu}
\end{equation}

By Theorem~\mref{thm:polyphi}.(\mref{it:polyphi}), taking $\vep \to 0$ in Eq.~(\mref {eq:1expstu}) gives
\begin{equation}
Z^{RB}(x_1^\ell\frakx_{\vec s})=\sum_{1\leq j_i\leq m_i, 0\leq i\leq\ell} a_{i,j_i}Z^{RB}(x_1^i)\zeta (\vec s_{i,j_i}).
\mlabel{eq:1shf}
\end{equation}
On the other hand, by Lemma~\mref{lem:zsum}, the Laurent series expansions of the regularized MZVs $Z(\wvec{\vec{u}}{\vec{u}};\vep)$ in Eq.~(\mref{eq:1expshu}) are in $\CC[T]\{\{\vep\}\}$. Thus the corresponding renormalized values $\zeta(\wvec{\vec{u}}{\vec{u}})$ defined in~\cite[Definition 3.5]{GZ} are obtained by taking $\vep=0$ in $Z(\wvec{\vec{u}}{\vec{u}};\vep)$.
Thus we have
\begin{equation}
\zeta(\wvec{\{1\}^\ell,\vec s}{\{1\}^\ell,\vec s})=\sum_{1\leq j_i\leq m_i, 0\leq i\leq\ell} a_{i,j_i}\zeta(\wvec{\{1\}^i}{\{1\}^i})\zeta (\vec s_{i,j_i}).
\mlabel{eq:1qsh}
\end{equation}
Note that $x_1^\ell=x_1^{\ssha \ell}/\ell!$ in $\calh\shf\lone\nonu$. Thus with the assignment
\begin{equation}
\beta (\frac {T^\ell}{\ell!})=\beta (Z^{RB}(x_1^\ell):=\zeta (\wvec{\{1\}^\ell}{\{1\}^\ell}), \quad \ell\geq 1,
\mlabel{eq:beta}
\end{equation}
and $\CC$-linearity, from Eqs.~(\mref{eq:1shf}) and (\mref{eq:1qsh}) we have
\begin{equation}
\beta (Z^{RB}(x_1^\ell \frakx_{\vec s}))=\zeta(\wvec{\{1\}^\ell,\vec s}{\{1\}^\ell,\vec s}),
\mlabel{eq:beta2}
\end{equation}
giving a linear map
$$ \beta: \CC[T] \to \CC[T].$$

\begin {theorem}
We have $$\beta =\rho ^{-1}$$ for the $\rho$ in Eq.~(\mref{eq:rho})
from~\mcite{IKZ}. \mlabel{thm:ed2}
\end{theorem}

\begin {proof}
By Theorem~\mref{thm:rbmzv} we have $Z^{RB}=Z^{\ssha}$.
By Theorem 4.5 and Proposition 4.7 in~\mcite{GZ}, we have
$\zeta(\wvec{\vec{u}}{\vec{u}})=Z^*(\vec{u})$ for $\vec{u}\in \ZZ^k_{\geq 1}$. Thus by Eq.~(\mref{eq:beta2}) and Theorem 1 in \mcite{IKZ}, $\beta$ agrees with $\rho^{-1}$.
\end {proof}

We end our discussion with an application of Theorem~\mref{thm:ed2}.
{}From the property of $\rho$ in~(\mref{eq:rho}):
$$\rho (e^{Tu})=A(u)e^{Tu}
$$
and Theorem~\mref{thm:ed2}, we have
$$\frac 1{A(u)}e^{Tu}=\beta (e^{Tu}).
$$
But by the definition of $A(u)$ in Eq.~(\mref{eq:A}) and the identification of $T$ with $Z^*(1)$, we have
$$\frac 1{A(u)}e^{Tu}=\exp(\sum _{n=1}^\infty (-1)^{n-1}Z ^*(n)\frac {u^n}n).
$$
By Eq.~(\mref{eq:beta}), we have
$$\beta (e^{Tu})=1+\sum_{n=1}^\infty Z^*(\{1\}^n)u^n.
$$
Therefore we have
\begin{coro}
$$\exp(\sum _{n=1}^\infty (-1)^{n-1}Z^*(n)\frac {u^n}n)=1+\sum_{n=1}^\infty Z^*(\{1\}^n)u^n.
$$
\mlabel{co:zetaone}
\end{coro}
This is an extension of the well-known formula~\mcite{IKZ}
$$
\exp(\sum _{n=1}^\infty (-1)^{n-1}\zeta ^*(nk)\frac {u^n}n)=1+\sum_{n=1}^\infty \zeta
^*(\{k\}^n)u^n, \quad k\geq 2.
$$
and can also be derived from~\cite[(5.8)]{IKZ}.


\end{document}